\documentclass[11pt]{amsart}

\usepackage{amsfonts,amssymb}
\usepackage{enumerate,graphicx}
\usepackage{subcaption}
\usepackage[square,numbers]{natbib}
\usepackage{tikz,pstricks-add}
\usepackage{MnSymbol}
\usepackage{mathrsfs,bbm}
\usepackage[linesnumbered,lined,boxed,commentsnumbered]{algorithm2e}
\newtheorem{theorem}{Theorem}[section]

\newtheorem{proposition}[theorem]{Proposition}

\newtheorem{lemma}[theorem]{Lemma}

\theoremstyle{definition}

\newtheorem{definition}[theorem]{Definition}
\newtheorem{example}[theorem]{Example}

\newtheorem{problem}{Problem}
\newtheorem{remark}[theorem]{Remark}

\newcommand\norm[1]{\left\lvert#1\right\rvert}

\newcommand\edcedge[3]{#1 \rightarrow (#2, #3)}

\newcommand\ket[1]{\langle #1 \rangle}
\usetikzlibrary{calc}
\pgfdeclarelayer{bg}
\pgfsetlayers{bg,main}

\begin{document}

\title{Decidability of Irreducible tree shifts of finite type}

%\keywords{Topological degree; entropy dimension; free group; finitely generated group; Cayley graph; conjugacy-invariant}
%\subjclass[2010]{Primary 37A35, 37B10}

\author{Jung-Chao Ban}
\address[Jung-Chao Ban]{Department of Mathematical Sciences, National Chengchi University, Taipei 11605, Taiwan, ROC.}
\address{Math. Division, National Center for Theoretical Science, National Taiwan University, Taipei 10617, Taiwan. ROC.}
\email{jcban@nccu.edu.tw}

\author[Chih-Hung Chang]{Chih-Hung Chang*}
\thanks{*Author to whom any correspondence should be addressed.} %To whom correspondence should be addressed
\address[Chih-Hung Chang]{Department of Applied Mathematics, National University of Kaohsiung, Kaohsiung 81148, Taiwan, ROC.}
\email{chchang@nuk.edu.tw}

\author{Nai-Zhu Huang}
\address[Nai-Zhu Huang]{Department of Applied Mathematics, National Chiao Tung University, Hsinchu 30010, Taiwan, ROC.}
\email{naizhu7@gmail.com}

\author{Yu-Liang Wu}
\address[Yu-Liang Wu]{Department of Electrical and Computer Engineering, National Chiao Tung University, Hsinchu 30010, Taiwan, ROC.}
\email{s92077.eed04@g2.nctu.edu.tw}

\thanks{This work is partially supported by the Ministry of Science and Technology, ROC (Contract No MOST 107-2115-M-259 -001 -MY2 and 107-2115-M-390 -002 -MY2).}
\date{}

\baselineskip=1.2\baselineskip

% -------------------------------------------------------------
\begin{abstract}
We reveal an algorithm for determining the complete prefix code irreducibility (CPC-irreducibility) of dyadic trees labeled by a finite alphabet. By introducing an extended directed graph representation of tree shift of finite type (TSFT), we show that the CPC-irreducibility of TSFTs is related to the connectivity of its graph representation, which is a similar result to one-dimensional shifts of finite type.
\end{abstract}

\maketitle

% ------------------------------------------------------
\section{Introduction}

Tree shifts, introduced by Aubrun and B\'{e}al \cite{AB-TCS2012, AB-TCS2013}, are shift spaces over Cayley trees. They are more complicated than one-dimensional shift spaces while still possess a natural one-dimensional structure of symbolic dynamical systems equipped with multiple shift maps. In classical symbolic dynamical systems of one-dimension, shifts of finite type (SFTs) play a fundamental and an essential role, and the investigation into their graph representations uncovers crucial properties such as irreducibility, mixing, and the existence of periodic points (cf.~\cite{Dev-1987, LM-1995}). We list some well-studied properties below. An SFT is nonempty if and only if its essential graph representation contains a cycle. Every nonempty SFT contains periodic points and an irreducible SFT has dense periodic points (see \cite{Kit-1998, LM-1995}). Nevertheless, when dealing with multidimensional shift spaces, contrary results have been obtained.

Firstly, the emptiness problem is undecidable for two-dimensional SFTs; there is an aperiodic SFT which has positive topological entropy, and there is a nonempty SFT which exhibits nonextensible local patterns \cite{Berger-MAMS1966, CulikII-DM1996, GK-SMJ1972, Kari-DM1996, Robinson-IM1971, Schmidt-2001}. These results reveal the uncertainty of multidimensional shift spaces and have attracted much attention. Recently, Sharma and Kumar \cite{SK-2016} demonstrated the necessary and sufficient condition for determining if a multidimensional SFT is empty and went further to provide a sufficient condition for multidimensional SFTs exhibiting periodic points. More precisely, they used a paricular irreducibility and mixing conditions to guarantee the nonemptiness, as well as the denseness of periodic points, of shift spaces. Boyle \emph{et al.}~\cite{BPS-TAMS2010} introduced a mixing condition known as \emph{block gluing} and showed that every two-dimensional block gluing SFT has dense periodic points; however, the denseness of periodic points in general multidimensional block-gluing SFTs is yet to be determined. Besides, there is the lack of an algorithm for determining if a shift space contains dense periodic points since it is undecidable. (Note that containing dense periodic points is a necessary condition of chaos in the sense of Devaney, much effort has  been put into finding the criteria for it.) Chandgotia and Marcus \cite{CM-PJM2018} provided a sufficient condition for block gluing shift spaces that are derived from an undirected connected graph, pointing out the key role played by the finiteness of the diameter of the corresponding graph. As there are weakly and strongly periodic points in multidimensional shifts (a weakly periodic point $x$ of a $\mathbb{Z}^n$-SFT is a point that satisfies $\sigma_u x = x$ for some $0 \neq u \in \mathbb{Z}^n$), there exists $\mathbb{Z}^n$-SFT which has no weakly periodic point \cite{CK-JUCS1995}. Since it is difficult to verify the existence of weakly/strongly periodic points for general multidimensional shifts, mixing property is crucial in solving such problems. For more details on the recent works in multidimensional shift spaces, the readers are referred to \cite{BHL+-2015, BFM-IJAC2005, BPS-TAMS2010, Briceno-ETDS2016, JM-ETDS2005, Lightwood-ETDS2003, MP-1979, MP-PLMS1981, MP-2016, PS-TAMS2015, Schraudner-DCDS2010, Ward-IMN1994} and the references therein. While it is acknowledged that stronger mixing properties yield a better structure of the systems, the examination of mixing properties remains a challenge.

Another important problem in symbolic dynamics is the classification of shift spaces. While the conjugacy of one-sided shifts of finite type is decidable \cite{Wil-AoM1973}, the conjugacy of multidimensional SFTs is not (see \cite{CJJ+-ETDS2003, JP-JCSS2015, JM-PAMS1999, LS-2002} for instance). Aubrun and B\'{e}al introduced the so-called \emph{CPC-irreducible tree shifts} (defined in Section 4) and provided an algorithm for determining the conjugacy between CPC-irreducible TSFTs \cite{AB-TCS2012}. In addition, they addressed the issue concerning the existence of CPC-irreducible sofic tree shifts which are not the factors of TSFTs \cite{AB-TCS2012, AB-TCS2013}; this is inconsistent with the classical one-dimensional case where an irreducible sofic shift is the factor of some irreducible SFT \cite{LM-1995}. Meanwhile, the domino problem (also known as the emptiness problem) is undecidable on surface groups (cf.~\cite{ABM-2018}). Piantadosi \cite{Piantadosi-DCDS2008} indicated that every SFT over finitely generated free group $G$ has a weakly periodic point and there is a $G$-SFT which has no strongly periodic point. It is of interest that under what condition a TSFT must contain strongly periodic points. An affirmative result was obtained by Ban and Chang \cite{BC-TAMS2017, BC-JMP2017}, demonstrating that every CPC-irreducible TSFT contains dense \emph{CPC-periodic points} (defined in Section 4) with a CPC-periodic point being strongly periodic. (We remark that Ceccherini-Silberstein \emph{et al.} \cite{CCF+-TCS2013} also demonstrated that strongly periodic points are dense in all sofic tree shifts.) The question of the existence of an algorithm which determines the CPC-irreducibility of TSFTs subsequently follows. Utilizing graph representation of TSFT introduced in \cite{BC-JMP2017} for the emptiness problem, the study of irreducibility is then extended from one-dimensional SFTs to TSFTs.% Apart from those investigation, Ban and Chang \cite{BC-JMP2017} introduce the graph representation of TSFT and show that a TSFT is nonempty if and only if has an essential graph representation. which extends the elucidation of graph representation of one-dimensional SFTs to TSFTs. %Recently, the complexity of TSFTs has been elucidated by Ban and Chang \cite{BC-N2017} and Petersen and Salama \cite{PS-TCS2018} separately.

\begin{table}
\begin{tabular}{cccc}
			\textbf{Space} & $\mathbb{Z}$-SFTs & $\mathbb{Z}^n$-SFTs & TSFTs \\
			\hline
			\textbf{Specification} & Irreducible & Irreducible & CPC-irreducible \\
		%	\hline
			\textbf{Decidability}& Decidable & Undecidable & Decidable% \\
%			& (Graphs) &  & (Extended Graphs)
\end{tabular}
\caption{Decidability of irreducibility of shifts of finite type.}
\label{table:decide-irr}
\end{table}

In this paper, we demonstrate that CPC-irreducibility of TSFTs is decidable and derive an algorithm for the examination of CPC-irreducibility. The difference between CPC-irreducibility and classical irreducibility is that CPC-irreducibility builds a wall-like cross-section for a given pattern with the second pattern necessarily sticking to the entire ``wall'', where as the classical irreducibility requires only that any two given patterns can be connected. We introduce \emph{extended directed graph representation} (defined in Section 5) of tree shifts of finite type, which is an extension of the classical graph representation of shifts of finite type (cf.~\cite{LM-1995}) and graph representation of TSFT introduced by Ban and Chang (cf.~\cite{BC-TAMS2017}). More specifically, an extended directed graph contains a set consisting of \emph{divergent-edges}, reflecting the structure of the tree and local patterns. The divergent-edge set is the main difference between graph representations of SFTs and TSFTs, and plays a crucial role in the determination of CPC-irreducibility. After a necessary repeated process of reduction, a tree shift of finite type is CPC-irreducible if and only if its extended graph representation is strongly connected. Table \ref{table:decide-irr} lists the decidability of the irreducibility of shifts of finite type over different underlying lattices.

The paper is organized as follow. We first reiterate the definitions of tree shifts that are relevant to the analysis here in Section 2. In Section 3, properties of complete prefix codes are introduced whereas in Section 4 those relations between CPC-irreducible TSFTs that are necessary in deriving the main results are presented. We next introduce extended directed graph representation of TSFTs in Section 5 and in Section 6 we demonstrate that the CPC-irreducibility of TSFTs is decidable. Finally, the flowchart of the algorithm is presented together with a brief discussion and open problems in the concluding section.

% ------------------------------------------------------
\section{Notation and Terminology}

Despite most of what we derive extends to general trees, we focus on labelings of the infinite dyadic tree, which is the set of all finite words on a two-element alphabet $\Sigma = \{s_1, s_2\}$. Algebraically speaking, the infinite dyadic tree $\Sigma^* = \bigcup\limits_{n \geq 0} \Sigma^n$ is a free monoid, where $\Sigma^n$ denotes the set of all finite words of length $n$ and $\Sigma^0 = \{\epsilon\}$ consists of the empty word. A word $g \in \Sigma^*$ corresponds uniquely to a node of the tree. We denote by $|g|$ the length of the word $g$.

Let $\mathcal{A}$ be a finite labeling set. A \emph{labeled tree} (or \emph{configuration}) is a function $t: \Sigma^* \to \mathcal{A}$. For each $g \in \Sigma^*$, $t_g := t(g)$ is the label attached to the node determined by $g$. We denote by $\mathcal{T}$ (or $\mathcal{A}^{\Sigma^*}$) the set of all labeled trees on $\mathcal{A}$. The shift transformation $\sigma: \Sigma^* \times \mathcal{T} \to \mathcal{T}$ is defined by $(\sigma_w t)_g := \sigma(w, t)_g = t_{wg}$ for all $w, g \in \Sigma^*$. For each $n \geq 0$, let $\Delta_n = \bigcup\limits_{0 \leq i \leq n} \Sigma^i$ denote the initial $n$-subtree of the dyadic tree. Note that $\Delta_n$ has $n+1$ levels. An \emph{$n$-block} $u$ is a labeling of the $n$-subtree $\Delta_n$, and $\Delta_n$, which is denoted by $s(u)$, is called the \emph{support} of $u$. We say that an $n$-block $u$ \emph{appears} in a labeled tree $t$ (or $u$ \emph{is accepted} by $t$) if there is a node $g \in \Sigma^*$ such that $t_{gw} = u_w$ for all $w \in \Delta_n$. A \emph{tree shift} $X$ is the set of all labeled trees which avoid all of a certain set of blocks (such a set is called a \emph{forbidden set} of $X$). We denote by $X = \mathsf{X}_{\mathcal{F}}$. A tree shift $X$ is called a \emph{tree shift of finite type} (TSFT) if $X = \mathsf{X}_{\mathcal{F}}$ for some finite forbidden set $\mathcal{F}$.

Suppose $u$ is a $1$-block for which $u_{\epsilon} = \alpha, u_{s_1} = \beta$, and $u_{s_2} = \gamma$. We may write $u$ as $\edcedge{\alpha}{\beta}{\gamma}$ for convenience. Furthermore, we denote by $\partial \Delta_n$ the \emph{boundary} of the initial $n$-subtree; that is,
$$
\partial \Delta_n = \{g \in \Delta_n: g s_i \notin \Delta_n \text{ for } i = 1, 2\} = \{g \in \Delta_n: |g| = n\}.
$$
A subset $S$ of $\Sigma^*$ is called \emph{prefix-closed} if all prefixes of $S$ are in $S$, and the boundary of $S$ is defined similarly as above; that is,
$$
\partial S = \{g \in S: g s_i \notin S \text{ for } i = 1, 2\}.
$$
A finite subset $S \subseteq \Sigma^*$ is called a \emph{prefix code} if no word in $S$ is a prefix of another word in $S$; a prefix code $S$ is called a \emph{complete prefix code} (CPC) if for every $w \in \Sigma^*$ with $|w| \geq \max\{|g|: g \in S\}$ there exists $g \in S$ such that $g$ is a prefix of $w$. A CPC forms a sort of cross-section of the tree such that each infinite path from the root intersects with the CPC at exactly one of its nodes.

% ------------------------------------------------------
\section{Properties of Complete Prefix Code}

In this section, we reveal properties of complete prefix code for the self-containedness of this paper. For more details, we refer the readers to \cite{MacKay-2003}.

Suppose $S$ is a CPC. Define
$$
R(S) = \{g \in \Sigma^*: g g' \in S \text{ for some } g' \in \Sigma^*\}
$$
be the initial finite subtree whose boundary is $S$.

\begin{proposition} \label{prop:CPC4_0}
Let $S$ be a CPC. For each $\overline{g} \in R(S)$, $S' := \partial A$ is a CPC and $R(S') = A$, where $A =\{h \in \Sigma^*: \overline{g} h \in R(S)\}$.
\end{proposition}
\begin{proof}
First, $S'$ is a prefix code, for if otherwise, there exist some $g, g' \in S'$ such that $g \ne g'$ and $g'$ is a prefix of $g$ and thus $\overline{g} g' \in S$ is a prefix of $\overline{g} g \in S$, which contradicts that $S$ is a CPC. Herein, we refer to proper prefix as prefix unless otherwise stated, and $g$ is a proper prefix of $h$ means that $h = g g'$ for some $g' \neq \epsilon$.
    
Next, we show that $S'$ is a complete prefix code. That is, for each $g \in \Sigma^*$ with $\norm{g} \ge \max_{h \in S'} \norm{h}$, there exists some prefix $g' \in S'$ of $g$. Note that, for sufficiently large $k \in \mathbb{N}$, $\overline{g} g s_1^k$ has a prefix $r \in S$ by the assumption that $S$ is a CPC. In this case, we can show that $\norm{\overline{g} g} \ge \norm{r}$. For if $\norm{\overline{g} g} < \norm{r}$ along with that $r \in S = \partial R(S)$ and that $R(S)$ is prefix-closed, it implies that $\norm{g} < \max_{h \in S'} \norm{h}$. It therefore contradicts that $\norm{g} \ge \max_{h \in S'} \norm{h}$. Since $r \in S$, $r$ cannot be a prefix of $\overline{g}$. Therefore, $r = \overline{g} g'$ for some $g' \in A$. This indicates that $g'$ is a prefix of $g$. Since $\overline{g} g' s_i \notin R(S)$ for all $s_i \in \Sigma$, it follows that $g' \in S'$. 
    
Finally, it is left to show that $R(S') = A$. If $h \in R(S')$, then there exists some $h' \in S'$ where $h$ is a prefix of $h'$. Hence, $\overline{g} h \in R(S)$ is a prefix of $\overline{g} h' \in S$ and $h \in A$ as a consequence. If $h \in A$, then $\overline{g} h \in R(S)$ is a prefix of some $\overline{g} h' \in S = \partial R(S)$, i.e., $h' \in \partial R(S')$. Since $h$ is a prefix of $h'$, it arrives at $h \in R(S')$.
\end{proof}

Proposition \ref{prop:CPC4_0} demonstrates that if $S$ is a CPC and $\overline{g} \in R(S)$, then $\{g \in \Sigma^*: \overline{g} g \in S\}$ is also a CPC. In addition, let $S_{\overline{g}} = \{g \in S: \overline{g} \text{ is a prefix of } g\}$, we show that replacing $S_{\overline{g}}$ with any CPC remains a CPC.

\begin{proposition}
Let $S_1$, $S_2$ be CPCs, and $\overline{g} \in R(S_1)$. Suppose $S = S_1 \setminus \{\overline{g} h: h \in \Sigma^*\} \bigcup \overline{g} S_2$. Then,
\begin{enumerate}
		\item $S$ is a CPC.
        \item $R(S) = (R(S_1) \setminus \{\overline{g} h \in R(S_1)\}) \bigcup \overline{g}  R(S_2)$.
\end{enumerate}
\end{proposition}
\begin{proof}
	(1) Firstly, we show that $S$ is a prefix code by contradiction. That is, there exists some $g', g'' \in S$ with $g' \ne g''$ such that $g'$ is a prefix of $g''$. Then, it must lies in the following two cases:
    \begin{itemize}
    	\item $\overline{g}$ is a prefix of $g'$. Then, $g' = \overline{g} h'$ and $g'' = \overline{g} h''$ for some $h', h'' \in \Sigma^*$, where $h', h'' \in S_2$ and $h'$ is a prefix of $h''$. It contradicts that $S_2$ is a prefix code.
        \item $\overline{g}$ is not a prefix of $g'$. Then, $\overline{g}$ is not a prefix of $g''$. Hence, $g', g'' \in S_1$ contradicts that $S_1$ is a prefix code.
    \end{itemize}
    
    Next, we show that $S$ is a CPC. That is, for each $g \in \Sigma^*$ with $\norm{g} \ge \max_{h \in S} \norm{h}$, there exists some $g' \in S$ such that $g'$ is a prefix of $g$. It can be verified by considering the following two cases:
    \begin{itemize}
    	\item $\overline{g}$ is a prefix of $g$. Then, $g = \overline{g} h'$, where $\norm{\overline{g}} + \norm{h'} = \norm{g} \ge \max_{h \in S} \norm{h} \ge \norm{\overline{g}} + \max_{h \in S_2} \norm{h}$.  Thus, there exists some $h'' \in S_2$ such that $h''$ is a prefix of $h'$ and thus $g':=\overline{g} h'' \in S$ is a prefix of $g$.
        \item $\overline{g}$ is not a prefix of $g$. Then for $s_1 \in \Sigma$ and sufficiently large $k \in \mathbb{N}$, $g s_1^k$ has a prefix $g' \in S_1$. In this case, $g' \in S$. Furthermore, $g'$ is a prefix of $g$, for if otherwise, $\norm{g'} > \norm{g} \ge \max_{h \in S} \norm{h}$ which results in a contradiction.
    \end{itemize}
    
    (2) For convenience, denote $A = (R(S_1) \setminus \{\overline{g} h \in R(S_1)\}) \bigcup \overline{g}  R(S_2)$. 
    
    If $g \in R(S)$, then it must be in one of the following two cases:
    \begin{itemize}
    	\item $\overline{g}$ is a prefix of $g$. In this case, there is some $g' \in \overline{g} S_2$ such that $g$ is the prefix of $g'$ and thus $g \in \overline{g} R(S_2) \subset A$.
        \item $\overline{g}$ is not a prefix of $g$, then $g$ is a prefix of some $g' \in S_1 \setminus \{\overline{g} h: h \in \Sigma^*\}$. Thus, $g \in R(S_1)$ and $g \notin \overline{g} R(S_2)$. Hence, $g \in A$.
    \end{itemize}
    
    On the other hand, if $g \in A$, then it must be in one of the following two cases:
    \begin{itemize}
    	\item $g \in \overline{g} R(S_2)$. Then, $g \in R(S)$ naturally.
        \item $g \in R(S_1) \setminus \{\overline{g} h \in R(S_1)\}$. Then, $g$ is a prefix of some $g' \in R(S_1)$. If $\overline{g}$ is not a prefix of $g'$, then $g' \in R(S_1 \setminus \{\overline{g} h \in R(S_1)\}) \subset R(S)$. If $\overline{g}$ is a prefix of $g'$, then there exists some $g' \in \overline{g} S_2$ such that $g$ is a prefix of $g'$ and thus $g \in R(S)$.
    \end{itemize}
    
The discussion above leads to that $R(S) \subset A$ and $A \subset R(S)$, which completes the proof.
\end{proof}

\begin{proposition} \label{prop:CPC4}
	Let $S_1$, $S_2$ be CPCs, $\overline{g} \in R(S_1)$. Define $P=(R(S_1) \setminus \{\overline{g} h \in R(S_1)\}) \bigcup \{\overline{g} h: h \in R(S_2)\}$. Then,
	\begin{enumerate}
		\item $P$ is prefix-closed. \label{prop:CPC4-1}
		\item $\partial P$ is a CPC. \label{prop:CPC4-2}
	\end{enumerate}
\end{proposition}
\begin{proof}
	({\ref{prop:CPC4-1}}) If $g = g_1 \ldots g_{k-1} g_k \in R(S_1) \setminus \{\overline{g} h \in R(S_1) \}$, then for any prefix $g':=g_1 \ldots g_{m}$ of $g$ for some $m < k$, it is clear that $g' \in R(S_1)$ since $R(S_1)$ is prefix-closed. On the other hand, $g' \notin \{\overline{g} h \in R(S_1) \}$, for if otherwise, $g'\in \{\overline{g} h \in R(S_1) \}$ and so $g \in \{\overline{g} h \in R(S_1) \}$. Hence, $g'\in R(S_1) \setminus \{\overline{g} h \in R(S_1) \} \subset P$. 
    
	If $g \in \{\overline{g} h \in \Sigma^*: h \in R(S_2) \}$, denote $g=g_1 \ldots g_k = \overline{g} h'$ for some $h' \in \Sigma^*$. Then, any prefix $g'$ of $g$ must lie in one of the following two cases:
	\begin{itemize}
		\item $\overline{g}$ is a prefix of $g'$. Then, $g' = \overline{g} h''$ for some $h'' \in \Sigma^*$. Therefore, $h''$ is a prefix of $h'$ and $g' \in \{\overline{g} h \in \Sigma^*: h \in R(S_2) \} \subset P$.
		\item $\overline{g}$ is not a prefix of $g'$. Then, $g' \in R(S_1) \setminus \{\overline{g} h \in R(S_1) \} \subset P$ naturally.
	\end{itemize}
	({\ref{prop:CPC4-2}}) It can be shown that $\partial P$ is a prefix code. Otherwise, there are $g', g'' \in \partial P$ with $g' \ne g''$ yet $g''=g' g_1 \ldots g_k$ for some $g_i \in \Sigma$. Hence, $g' \notin \partial P$.
	
	Next, we show that $\partial P$ is a CPC. Suppose that $g \in \Sigma^*$ with $\norm{g} \ge \max_{h \in \partial P} \norm{h}$. There are two cases: 
	\begin{itemize}
		\item $\overline{g}$ is a prefix of $g$. Then, $g=\overline{g} h'$ for certain $h' \in \Sigma^*$. In this case $\norm{g} = \norm{\overline{g}} +\norm{h'} \ge \max_{h \in \partial P} \norm{h} \ge \norm{\overline{g}}+\max_{h \in S_2} \norm{h}$. This implies that there is some $w \in S_2$ with $w s_i \notin R(S_2)$ for all $s_i \in \Sigma$ such that $w$ is a prefix of $h'$. In this case, $\overline{g} w \in P$ and $\overline{g} w s_i \notin P$ for all $s_i \in \Sigma$. Hence, $\overline{g} w \in \partial P$ is a prefix of $g = \overline{g} h'$.
		\item $\overline{g}$ is not a prefix of $g$. 
		\begin{itemize}
			\item If $\norm{g} \ge \max_{h \in R(S_1)} \norm{h}$, then there exists $g' \in S_1$ such that $g'$ is a prefix of $g$ and $g' s_i \notin R(S_1)$. Note that $\overline{g}$ is not a prefix of $g$ and thus is not a prefix of $g'$ or $g' s_i$ for all $s_i \in \Sigma$, since $\norm{g} \ge \norm{g'}$. Hence, $g' \in P$ and $g' s_i \notin P$ for each $s_i \in \Sigma$. 
			\item If $\max_{h \in \partial P} \norm{h} \le \norm{g} < \max_{h \in R(S_1)} \norm{h}$, then there exists some $k>0$ such that $g s_i^k$ has a prefix $g' \in S_1$. In this case, $g'$ is a prefix of $g$. Otherwise, $g'=g s_i^m$ for some $m$ with $1 \le m \le k$ and $g' \in S_1$ implies $g s_i \in R(S_1) \bigcap P$ and thus $\norm{g} < \max_{h \in \partial P} \norm{h}$. Hence, $g' s_i \notin R(S_1)$. We have derived that $g' \in P$ and $g' s_i \notin P$ for each $s_i \in \Sigma$.
		\end{itemize}
	\end{itemize}
\end{proof}

\begin{proposition} \label{prop:CPC5}
Let $S, P$ be CPCs. Then $S P$ is a CPC.
\end{proposition}
\begin{proof}
	Firstly, we claim that $S P$ is a prefix code. Otherwise, there exist $g, h \in S P$, such $h \ne g$ is a prefix of $g$, where $g = g' g''$, $h = h' h''$, $g', h' \in S$, and $g'', h'' \in P$. In this case, $h'$ is a prefix of $g' g''$, and thus either $h'$ is a prefix of $g'$ or $g'$ is a prefix of $h'$. It implies $g' = h'$ for the reason that $g', h' \in S$ and $S$ is a CPC. Hence, $h''$ is a prefix of $g''$. Since $g'', h'' \in P$ and $P$ is a CPC, $g''=h''$. This contradicts that $g \ne h$.
    
	Now we claim that $S P$ is a CPC. That is, given any $g \in \Sigma^*$ with $\norm{g} \ge \max_{h \in S P} \norm{h}$, $g$ has a prefix $g' \in S$. Since $\norm{g} \ge \max_{h \in S P} \norm{h} \ge \max_{h \in S} \norm{h}$, we may denote $g = g' r$ for some $g' \in S$, $r \in \Sigma^*$. Then for some sufficiently large $k  \in \mathbb{N}$ such that $\norm{r s_1^k} \ge \max_{h \in P} \norm{h}$, $r s_1^{k}$ has a prefix $h' \in P$. In this case, either $h'$ is a prefix of $r$ or $r$ is a prefix of $h'$. Note that if $h'$ is not a prefix of $r$, then $\norm{g' r} < \norm{g'h'} \le \max_{h \in S P} \norm{h}$, which contradicts that $\norm{g} = \norm{g' r} \ge \max_{h \in S P} \norm{h}$.
\end{proof}

% ------------------------------------------------------
\section{Irreducible on Complete Prefix Code}

Suppose $X$ is a tree shift of finite type. Then there exists $\mathcal{F} \subset \mathcal{A}^{\Delta_n}$ such that $X = \mathsf{X}_{\mathcal{F}}$ for some $n \in \mathbb{N}$. Ban and Chang \cite{BC-TAMS2017} showed that there exist $\mathcal{A}'$ and $\mathcal{F}' \subset (\mathcal{A}')^{\Delta_1}$ such that $\mathsf{X}_{\mathcal{F}}$ is topologically conjugate with $\mathsf{X}_{\mathcal{F}'}$, which is analogous to the classical result. For the rest of this elucidation, the forbidden set $\mathcal{F}$ of a TSFT $X = \mathsf{X}_{\mathcal{F}}$ is referred to as a subset of $\mathcal{A}^{\Delta_1}$ unless stated otherwise; such an $\mathcal{F}$ is ``maximal'' in the sense that every pattern $u \in \mathcal{A}^{\Delta_1} \setminus \mathcal{F}$ is extensible. In addition, we say that $X$ is induced by an allowable set $B = \mathcal{A}^{\Delta_1} \setminus \mathcal{F}$.

\begin{proposition} \label{prop:extensible}
Let $X$ be a TSFT induced by some allowable set $B \subseteq \mathcal{A}^{\Delta_1}$. Suppose $u$ is a pattern with following properties:
\begin{itemize}
	\item $\partial s(u)$ is a CPC.
	\item If $g \Delta_1 \subseteq s(u)$, $\edcedge{u_g}{u_{g s_1}}{u_{g s_2}} \in B$.
\end{itemize}
Then, there exists an $x \in X$ such that $x|_{s(u)}=u$. That is, $u$ is an allowable pattern in $X$.
\end{proposition}
\begin{proof}
Suppose $M:=\min_{h \in \partial s(u)} \norm{h}$. Let $u^{(M)}:=u$. There exists a sequence of patterns $\{u^{(n)}\}_{n=M}^{\infty}$ such that
\begin{itemize}
	\item $n=\min_{h \in \partial s(u^{(n)})} \norm{h}$.
	\item $u^{(n)}|_{u^{(n-1)}} = u^{(n-1)}$ if $n>M$.
	\item $\partial s(u^{(n)})$ is a CPC.
	\item If $g \Delta_1 \subseteq s(u^{(n)})$, $\edcedge{u^{(n)}_g}{u^{(n)}_{g s_1}}{u^{(n)}_{g s_2}} \in B$.
\end{itemize}
This can be proved by construction on existence of $u^{(k+1)}$ given the finite sequence $\{u^{(n)}\}_{n=M}^{k}$ for all $k \ge M$.

When $k=M$, such properties are held by definition of $u^{(M)}$. 

Suppose the claim holds for $k$, the case $k+1$ can also be verified by induction hypothesis. Note that $\partial s(u^{(k)}) \Sigma$ is a CPC by Proposition \ref{prop:CPC5}, and
\begin{align*}
s(u^{(k+1)}) := R(\partial s(u^{(k)}) \Sigma) = s(u^{(k)}) \bigcup \partial  s(u^{(k)}) \Sigma.
\end{align*}
Alternatively, to construct $u^{(k+1)}$, it is sufficient to determine the pattern on $\partial s(u^{(k)}) \Sigma$. On the other hand, for each $g' \in \partial s(u^{(k)})$ there exists some $\edcedge{u^{(k)}_{g'}}{\alpha_{g', s_1}}{\alpha_{g', s_2}} \in B$ according to the assumption of essential graph representation. In this case, define $u^{(k+1)}$ as follows:
		\[u^{(k+1)}_g:=\begin{cases}
		u^{(k)}_g & \text{if } g \in s(u^{(k)}) \\
		\alpha_{g', s_i} & \text{if } g = g' s_i, s_i \in \Sigma \\
		\end{cases}\]
Then, $\min_{h \in \partial s(u^{(k+1)})} \norm{h} = 1+\min_{h \in \partial s(u^{(k)})} \norm{h}$, and $\partial s(u^{(k+1)})$ is a CPC by Proposition \ref{prop:CPC5}.

If $g \in \Sigma^*$ with $g \Delta_1 \subseteq s(u^{(k+1)})$, there are two cases:
\begin{itemize}
	\item $g \Delta_1 \subseteq s(u^{(k)})$. Then $\edcedge{u^{(k+1)}_g}{u^{(k+1)}_{g s_1}}{u^{(k+1)}_{g s_2}} = \edcedge{u^{(k)}_g}{u^{(k)}_{g s_1}}{u^{(k)}_{g s_2}} \in B$.
	\item $g \Delta_1 \nsubseteq s(u^{(k)})$. Then $g \in \partial s(u^{(k)})$ and $\edcedge{u^{(k+1)}_g}{u^{(k+1)}_{g s_1}}{u^{(k+1)}_{g s_2}} = \edcedge{u^{(k)}_g}{\alpha_{g, s_i}}{\alpha_{g, s_2}} \in B$.
\end{itemize}
Hence, the four properties are satisfied for the case of $k+1$.

By mathematical induction, the assertion holds for all $k \in \mathbb{N}$.

In this case, for $0 \le n < M$, define $u^{(n)}:=u^{(M)}|_{\Delta_{n-1}}$. Then, $\left\{ u^{(k)}\right\}_{k=1}^{\infty}$ determines an $x \in X$ by $x_g = u^{(\norm{g})}_g$ and $x_\epsilon = \alpha$. This completes the proof.
\end{proof}

A tree shift $X$ is called \emph{irreducible on complete prefix code} (CPC-irreducible) if for each pair of blocks $u, v \in B_n(X)$, there is an $x \in X$ and a complete prefix code $P \subset \bigcup_{k > n} \Sigma^k$ such that $u$ is a subtree of $x$ rooted at $\epsilon$ and $v$ is a subtree of $x$ rooted at $g$ for all $g \in P$, where $B_n(X)$ denotes the set of $n$-blocks of $X$. In other words, $x|_{\Delta_n} = u$ and $x|_{g \Delta_n} = v$ for each $g \in P$. CPC-irreducibility is defined by Aubrun and B\'{e}al and is named irreducible in \cite{AB-TCS2012, AB-TCS2013}. They extended the Williams' Classification Theorem to CPC-irreducible tree shifts of finite type. In addition, there exists a CPC-irreducible sofic tree shift which is not a factor of a CPC-irreducible TSFT.

Let $X$ be a tree shift and $x \in X$. We say that $x$ is \emph{strongly periodic} if the orbit $\{\sigma_g x\}_{g \in \Sigma^*}$ is finite. If there exists a CPC $P$ such that $\sigma_g x = x$ for each $g \in P$, then $x$ is a \emph{CPC-periodic point}. It is obvious that a CPC-periodic point is strongly periodic. Ban and Chang \cite{BC-TAMS2017} demonstrated that every CPC-irreducible TSFT has dense CPC-periodic points, which concludes that strongly periodic points are dense in CPC-irreducible TSFTs. Furthermore, they addressed the following equivalent statements of CPC-irreducibility. %B = \{1 \to (2, 1), 2 \to (2, 1)\} has strongly but no CPC-periodic points.

\begin{theorem}[See \cite{BC-TAMS2017}] \label{thm:equiv-def-irr}
Suppose $X$ is a tree shift. The following are equivalent.
\begin{enumerate}[\bf (i)]
\item $X$ is CPC-irreducible.
\item For each pair of blocks $u \in B_n(X), v \in B_m(X)$ there exists a collection of CPCs $\{P_w\}_{w \in \Sigma^n}$ and $x \in X$ such that
$$
x|_{\Delta_n} = u \quad \text{and} \quad x|_{w g \Delta_m} = v \text{ for all } w \in \Sigma^n, g \in P_w.
$$
\item For each pair of blocks $u \in B_n(X), v \in B_m(X)$ there exists a collection of CPCs $\{P_k\}_{1 \leq k \leq l}$ and $x \in X$ such that $x|_{\Delta_n} = u$ and, for each $w \in \Sigma^n$ there exists $1 \leq k \leq l$ such that $x|_{w g \Delta_m} = v$ for all $g \in P_k.$
\end{enumerate}
\end{theorem}

Whenever $X$ is a $1$-step TSFT, the CPC-irreducibility can be rephrased more elegantly. Note that $X$ is a $1$-step TSFT if $X = \mathsf{X}_{\mathcal{F}}$ for some $\mathcal{F} \subset \mathcal{A}^{\Delta_1}$.

\begin{lemma} \label{lem:CPC_irr}
Let $X$ be a $1$-step TSFT. Then $X$ is CPC-irreducible if and only if for every $\alpha, \beta \in \mathcal{A}$ there exists a CPC $S$ and $x \in X$ such that $x_\epsilon = \alpha$ and $x_g = \beta$ for each $g \in S$.
\end{lemma}
\begin{proof}
The proof follows directly from either the definition of CPC-irreducible tree shifts or from Theorem \ref{thm:equiv-def-irr}, thus it is omitted.
\end{proof}

We say that two symbols $\alpha$ and $\beta$ are \emph{CPC-connected (in $X$)} if there exists a CPC $S$ and $x \in X$ such that $x_\epsilon = \alpha$ and $x_g = \beta$ for all $g \in S$. Suppose $\edcedge{\alpha}{\beta}{\gamma}$ is an allowable block in $X$ for some $\beta \neq \gamma$. The following theorem demonstrates that replacing $\edcedge{\alpha}{\beta}{\gamma}$ or $\edcedge{\alpha}{\gamma}{\beta}$ with $\edcedge{\alpha}{\gamma}{\gamma}$ does not break the CPC-irreducibility if $\beta$ and $\gamma$ are CPC-connected. Figures \ref{fig:replacement-thm-X2Y} and \ref{fig:replacement-thm-Y2X} illustrate the idea of the proof of Theorem \ref{thm:replacement}.

\begin{figure}
\tikzset{
  r/.style = {draw, circle,fill = red!90!white, font=\bfseries, text =white},
  b/.style = {draw, circle,fill = blue!70!yellow, text = white},
  g/.style = {draw, circle,fill = green!90!white, font=\bfseries, text =white},
  c/.style = {draw, circle,fill = cyan!90!white, font=\bfseries, text =white},
  p/.style = {draw, circle,fill = pink!90!white, font=\bfseries, text =white},
  pu/.style = {draw, circle,fill = purple!90!white, font=\bfseries, text =white},
  ye/.style = {draw, circle,fill = yellow!90!white, font=\bfseries, text =white},
  itria/.style={draw,dashed,shape border uses incircle, isosceles triangle,shape border rotate=90,yshift=-1.5cm}
}
\begin{tikzpicture}[
    scale = 0.5, transform shape, thick,
    every node/.style = {minimum size = 5mm},
    grow = down,  % alignment of characters
    level 1/.style = {sibling distance=4.8cm},
    level 2/.style = {sibling distance=2.2cm}, 
    level 3/.style = {sibling distance=1.1cm},
    level 4/.style = {sibling distance=0.55cm},
    level distance = 1.5cm
  ]

\node[r]  (a) {$\alpha$} [grow'=down]
child {node (a1) [g] {$\beta$} 
 child{node[c] {$\beta_0$}}
 child{node[c] {$\beta_0$}}
}
child {node (a2) [b] {$\gamma$}
 child{node (a21) [pu] {$\delta$}
  child{node (a211)[c] {$\beta_0$}}
  child{node[c] {$\beta_0$}}
 }
 child{node (a22) [ye] {$\zeta$}
  child{node[c] {$\beta_0$}}
  child{node (a222) [c] {$\beta_0$}}
 }
};

\node at ($(a)+(0,-6)$) {\LARGE allowable pattern in $X$};

\node[r] at (15,0) (b) {$\alpha$} [grow'=down]
child {node (b1) [b] {$\gamma$}
 child{node (b11) [pu] {$\delta$}
  child{node (b111) [c] {$\beta_0$}}
  child{node (b112) [c] {$\beta_0$}}
 }
 child{node (b12) [ye] {$\zeta$}
  child{node (b121) [c] {$\beta_0$}}
  child{node (b122) [c] {$\beta_0$}}
 }
}
child {node (b2) [b] {$\gamma$}
 child{node (b21) [pu] {$\delta$}
  child{node (b211) [c] {$\beta_0$}}
  child{node (b212) [c] {$\beta_0$}}
 }
 child{node (b22) [ye] {$\zeta$}
  child{node (b221) [c] {$\beta_0$}}
  child{node (b222) [c] {$\beta_0$}}
 }
};

\node at ($(b)+(0,-6)$) {\LARGE allowable pattern in $Y$};

\draw[>=stealth,->] ($(a1)+(3.8cm,0)$) -- ($(b2)+(-3.8cm,0)$);
\draw[draw=none] (a1) -- (b2) node [midway, below]  {Replace};

\begin{pgfonlayer}{bg}
 \draw[ye] ($(a)$) -- ($(a1)$) -- ($(a2)$) -- cycle;
 \draw[g] ($(a2)$) -- ($(a21)$) -- ($(a211)$) -- ($(a222)$) -- ($(a22)$) -- cycle;
 \draw[p] ($(b)$) -- ($(b1)$) -- ($(b2)$) -- cycle;
 \draw[g] ($(b2)$) -- ($(b21)$) -- ($(b211)$) -- ($(b222)$) -- ($(b22)$) -- cycle;
 \draw[g] ($(b1)$) -- ($(b11)$) -- ($(b111)$) -- ($(b122)$) -- ($(b12)$) -- cycle;
\end{pgfonlayer}

\end{tikzpicture}
\caption{Illustration of proof of Theorem \ref{thm:replacement} Part 1: If $X$ is CPC-irreducible, then $Y$ is CPC-irreducible. To make the left pattern an allowable pattern in $Y$, we only need to replace the local pattern starting with $\beta$ by the one starting with $\gamma$.}
\label{fig:replacement-thm-X2Y}
\end{figure}
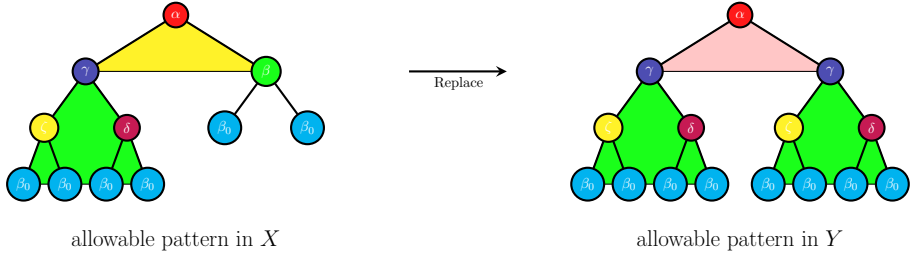

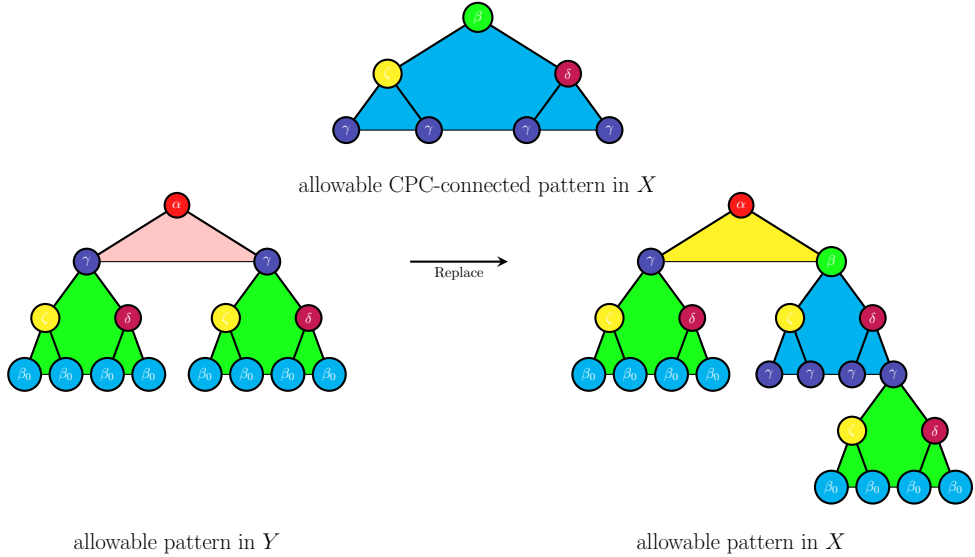
\begin{figure}
\tikzset{
  r/.style = {draw, circle,fill = red!90!white, font=\bfseries, text =white},
  b/.style = {draw, circle,fill = blue!70!yellow, text = white},
  g/.style = {draw, circle,fill = green!90!white, font=\bfseries, text =white},
  c/.style = {draw, circle,fill = cyan!90!white, font=\bfseries, text =white},
  p/.style = {draw, circle,fill = pink!90!white, font=\bfseries, text =white},
  pu/.style = {draw, circle,fill = purple!90!white, font=\bfseries, text =white},
  ye/.style = {draw, circle,fill = yellow!90!white, font=\bfseries, text =white},
  itria/.style={draw,dashed,shape border uses incircle, isosceles triangle,shape border rotate=90,yshift=-1.5cm}
}
\begin{tikzpicture}[
    scale = 0.5, transform shape, thick,
    every node/.style = {minimum size = 5mm},
    grow = down,  % alignment of characters
    level 1/.style = {sibling distance=4.8cm},
    level 2/.style = {sibling distance=2.2cm}, 
    level 3/.style = {sibling distance=1.1cm},
    level 4/.style = {sibling distance=2.2cm},
    level 5/.style = {sibling distance=1.1cm},
    level 6/.style = {sibling distance=0.55cm},
    level distance = 1.5cm
  ]

\node[g] at (8,0) (a) {$\beta$} [grow'=down]
child {node (a1) [pu] {$\delta$} 
 child{node (a11) [b] {$\gamma$}}
 child{node (a12) [b] {$\gamma$}}
}
child {node (a2) [ye] {$\zeta$}
 child{node (a21) [b] {$\gamma$}}
 child{node (a22) [b] {$\gamma$}}
};

\node at ($(a)+(0,-4.5)$) {\LARGE allowable CPC-connected pattern in $X$};
%\node at ($(a)+(7,-6)$) {\huge (a)};

\node[r] at (0,-5) (b) {$\alpha$} [grow'=down]
child {node (b1) [b] {$\gamma$}
 child{node (b11) [pu] {$\delta$}
  child{node (b111) [c] {$\beta_0$}}
  child{node (b112) [c] {$\beta_0$}}
 }
 child{node (b12) [ye] {$\zeta$}
  child{node (b121) [c] {$\beta_0$}}
  child{node (b122) [c] {$\beta_0$}}
 }
}
child {node (b2) [b] {$\gamma$}
 child{node (b21) [pu] {$\delta$}
  child{node (b211) [c] {$\beta_0$}}
  child{node (b212) [c] {$\beta_0$}}
 }
 child{node (b22) [ye] {$\zeta$}
  child{node (b221) [c] {$\beta_0$}}
  child{node (b222) [c] {$\beta_0$}}
 }
};

\node at ($(b)+(0,-9)$) {\LARGE allowable pattern in $Y$};

\node[r] at (15,-5) (c) {$\alpha$} [grow'=down]
child {node (c1) [g] {$\beta$}
 child {node (c11) [pu] {$\delta$} 
  child{node (c111) [b] {$\gamma$}
   child{node (c1111) [pu] {$\delta$}
    child{node (c11111) [c] {$\beta_0$}}
    child{node (c11112) [c] {$\beta_0$}}
   }
   child{node (c1112) [ye] {$\zeta$}
    child{node (c11121) [c] {$\beta_0$}}
    child{node (c11122) [c] {$\beta_0$}}
   }
  }
  child{node (c112) [b] {$\gamma$}}
 }
 child {node (c12) [ye] {$\zeta$}
  child{node (c121) [b] {$\gamma$}}
  child{node (c122) [b] {$\gamma$}}
 }
}
child {node (c2) [b] {$\gamma$}
 child{node (c21) [pu] {$\delta$}
  child{node (c211) [c] {$\beta_0$}}
  child{node (c212) [c] {$\beta_0$}}
 }
 child{node (c22) [ye] {$\zeta$}
  child{node (c221) [c] {$\beta_0$}}
  child{node (c222) [c] {$\beta_0$}}
 }
};

\node at ($(c)+(0,-9)$) {\LARGE allowable pattern in $X$};

\draw[>=stealth,->] ($(b1)+(3.8cm,0)$) -- ($(c2)+(-3.8cm,0)$);
\draw[draw=none] (b1) -- (c2) node [midway, below]  {Replace};

\begin{pgfonlayer}{bg}
 \draw[c] ($(a)$) -- ($(a1)$) -- ($(a11)$) -- ($(a22)$) -- ($(a2)$) -- cycle;
 \draw[p] ($(b)$) -- ($(b1)$) -- ($(b2)$) -- cycle;
 \draw[g] ($(b2)$) -- ($(b21)$) -- ($(b211)$) -- ($(b222)$) -- ($(b22)$) -- cycle;
 \draw[g] ($(b1)$) -- ($(b11)$) -- ($(b111)$) -- ($(b122)$) -- ($(b12)$) -- cycle;
 \draw[ye] ($(c)$) -- ($(c1)$) -- ($(c2)$) -- cycle;
 \draw[c] ($(c1)$) -- ($(c11)$) -- ($(c111)$) -- ($(c122)$) -- ($(c12)$) -- cycle;
 \draw[g] ($(c2)$) -- ($(c21)$) -- ($(c211)$) -- ($(c222)$) -- ($(c22)$) -- cycle;
 \draw[g] ($(c111)$) -- ($(c1111)$) -- ($(c11111)$) -- ($(c11122)$) -- ($(c1112)$) -- cycle;
\end{pgfonlayer}

\end{tikzpicture}
\caption{Illustration of proof of Theorem \ref{thm:replacement} Part 2: If $Y$ is CPC-irreducible, then $X$ is CPC-irreducible. Suppose the top pattern, which is allowable in $X$ such that $\beta$ is CPC-connected to $\gamma$. To make the lower left pattern an allowable pattern in $X$, we replace one local pattern starting with $\gamma$ by the top pattern first, then we complete the construction by gluing patterns starting with $\gamma$.}
\label{fig:replacement-thm-Y2X}
\end{figure}

\begin{theorem} \label{thm:replacement}
Let $X$ be a TSFT induced by some allowable set $B \subseteq \mathcal{A}^{\Delta_1}$. Suppose $\edcedge{\alpha}{\beta}{\gamma} \in B$ such that $\beta$ and $\gamma$ are CPC-connected in $X$. Let $Y$ be the TSFT induced by $B' := (B \setminus \{\edcedge{\alpha}{\beta}{\gamma}\}) \bigcup \{\edcedge{\alpha}{\gamma}{\gamma}\}$. Then, $X$ is CPC-irreducible if and only if $Y$ is CPC-irreducible.
\end{theorem}
\begin{proof}
Suppose $X$ is CPC-irreducible. Given any $\alpha_0, \beta_0 \in \mathcal{A}$, Lemma \ref{lem:CPC_irr} indicates that there exist $x = x^{\ket{\alpha_0, \beta_0}}$ and CPC $S = S^{\ket{\alpha_0, \beta_0}}$ such that $x_\epsilon = \alpha_0$ and $x_g = \beta_0$ for all $g \in S$. Let $u:=x|_{R(S)}$. We claim that there exists an allowable pattern $\overline{u}=\overline{u}^{\ket{\alpha_0,\beta_0}}$ in $Y$ such that $\overline{u}_\epsilon = \alpha_0$ and $\overline{u}_g = \beta_0$ for all $g \in \overline{S}$, where $\overline{S}( = \overline{S}^{\ket{\alpha_0,\beta_0}} ):= \partial s(\overline{u})$ is a CPC. Proposition \ref{prop:extensible} yields that $Y$ is CPC-irreducible.
		 
The desired $\overline{u}$ can be constructed as follows.
        
Denote $u^{(0)} = u$ and $S_0 = S$. Consider the set $A_0:=\{g \in \Sigma^*: g \Delta_1\ \subseteq R(S), \edcedge{u^{(0)}_{g}}{u^{(0)}_{g s_1}}{u^{(0)}_{g s_2}} = \edcedge{\alpha}{\beta}{\gamma}\}$. If $A_0 = \emptyset$, then $\overline{u} = u^{(0)}$. Otherwise, choose $\overline{g} \in A_0$ such that $\overline{g} = \max_{h \in A_0} \norm{h}$, and consider a specific pattern $u^{(1)}$ defined as
		\[u^{(1)}:= (u^{(0)} \setminus \{(\overline{g} s_1 h, u^{(0)}_{\overline{g} s_1 h}): \overline{g} s_1 h \in R(S_0)\} ) \bigcup \{(\overline{g} s_1 h, u^{(0)}_{\overline{g} s_2 h}): \overline{g} s_2 h \in R(S_0)\}. \]
%That is, $s(u^{(1)}) = s(u^{(0)})$, $u^{(1)}_{\overline{g} s_1 h} = u^{(0)}_{\overline{g} s_2 h}$ if $\overline{g} s_1 h \in R(S_0)$, and $u^{(1)}_g = u^{(0)}_g$ otherwise.
See Figure \ref{fig:replacement-thm-X2Y} for the construction of $u^{(1)}$. Herein, we use the set-theoretic definition of function to present the construction of desired pattern for clarity. Note that $S' := \partial \{h \in \Sigma^*: \overline{g} s_2 h \in R(S)\}$ is a CPC and $R(S') = \{h \in \Sigma^*: \overline{g} s_2 h \in R(S)\}$ by Proposition \ref{prop:CPC4_0}. Denote $S_1=\partial s(u^{(1)})$. It follows from Proposition \ref{prop:CPC4} that $S_1$ is a CPC.

Next, for any $g' \in S_1$, we claim that $u^{(1)}_{g'} = \beta_0$. If $\overline{g} s_1$ is not a prefix of $g'$, then $u^{(1)}_{g'} = u^{(0)}_{g'} = \beta_0$. If $\overline{g} s_1$ is a prefix of $g'$, $g' = \overline{g} s_1 w$, then $u^{(1)}_{\overline{g} s_1 w} = u^{(0)}_{\overline{g} s_2 w} = \beta_0$. That is, $u^{(1)}_g = \beta_0$ for all $g \in S$.
		
Note that $A_1:=\{g \in \Sigma^*: g \Delta_1\ \subseteq R(\overline{S}), \edcedge{u^{(0)}_{g}}{u^{(0)}_{g s_1}}{u^{(0)}_{g s_2}} = \edcedge{\alpha}{\beta}{\gamma} \} = A_0 \setminus \{\overline{g}\}$. If $A_1 = \emptyset$, then $\overline{u} = u^{(1)}$. Otherwise, we can construct $u^{(2)}$ via the same argument. Repeat the procedure and construct $\{A_n\}$ and $\{u^{(n)}\}$ with $A_n \supsetneq A_{n+1}$ if $A_n \ne \emptyset$. Let $N = \norm{A_0}$ and $\overline{u} = u^{(N)}$ is the desired pattern. 
		
Now if $Y$ is CPC-irreducible, a similar argument follows. Since $Y$ is CPC-irreducible, given any $\alpha_0, \beta_0 \in \mathcal{A}$, there exists some allowable pattern in $Y$, $u = u^{\ket{\alpha_0, \beta_0}}$ such that $S = S^{\ket{\alpha_0, \beta_0}} = \partial s(u)$ is a CPC and $u_\epsilon = \alpha_0$ and $u_g = \beta_0$ for all $g \in S$. It remains to find an allowable pattern in $X$, $\overline{u}=\overline{u}^{\ket{\alpha_0,\beta_0}}$, where $\overline{S}( = \overline{S}^{\ket{\alpha_0,\beta_0}} ):= \partial s(\overline{u})$ is a CPC, $\overline{u}_\epsilon = \alpha_0$ and $\overline{u}_g = \beta_0$ for all $g \in \overline{S}$. 
		
Note that if $\alpha_0 = \beta$ and $\beta_0 = \gamma$, existence of allowable pattern $\overline{u}^{\ket{\beta,\gamma}}$ is guaranteed by the assumption of CPC-connectedness in $X$ in Theorem \ref{thm:replacement}.
        
For other $(\alpha_0, \beta_0) \ne (\beta, \gamma)$, $\overline{u}$ can be constructed in the following manner.
        
Denote $u^{(0)} = u$ and $S_0 = S$. Let $A_0=(A^{\ket{\alpha_0,\beta_0}}):=\{g \in \Sigma^*: g \Delta_1\ \subseteq R(S), \edcedge{u_{g}}{u_{g s_1}}{u_{g s_2}} = \edcedge{\alpha}{\gamma}{\gamma}\}$. 
		
$A_0 \ne \emptyset$. Choose $\overline{g} \in A_0$ such that $\norm{\overline{g}} = \max_{h \in A_0} \norm{h}$. Consider the pattern $u^{(1)}$ defined as follows:
\begin{align*}
u^{(1)}:= &\quad(u^{(0)} \setminus \{(\overline{g} s_1 h, u^{(0)}_{\overline{g} s_1 h}): \overline{g} s_1 h \in R(S_0)\} ) \\
 &\bigcup \{(\overline{g} s_1 h, \overline{u}_{h}^{\ket{\beta,\gamma}})): h \in R(\overline{S}^{\ket{\beta,\gamma}})\} \\
 &\bigcup \{ (\overline{g} s_1 w h, u^{(0)}_{\overline{g} s_2 h}): w \in \overline{S}^{\ket{\beta,\gamma}}, \overline{g} s_2 h \in R(S_0) \},
\end{align*}
where $\overline{u}_{h}^{\ket{\beta,\gamma}}$ and $\overline{S}^{\ket{\beta,\gamma}}$ are already given. See Figure \ref{fig:replacement-thm-Y2X} for the construction of $u^{(1)}$.
        
Denote $S_1=\partial s(u^{(1)})$. Then, by Proposition \ref{prop:CPC4}, $S_1$ is a CPC.
		
Next, for all $g' \in S_1$, we claim that $u^{(1)}_{g'} = \beta_0$. If $\overline{g} s_1$ is not a prefix of $g'$, then $u^{(1)}_{g'} = u^{(0)}_{g'} = \beta_0$. If $\overline{g} s_1$ is a prefix of $g'$, $g' = \overline{g} s_1 w h$ where $w \in \overline{S}^{\ket{\beta,\gamma}}$. Thus, $u^{(1)}_{\overline{g} s_1 w h} = u^{(0)}_{\overline{g} s_2 h} = \beta_0$.
		
Note that $A_1:=\{g \in \Sigma^*: g \Delta_1\ \subseteq R(S_1), \edcedge{u^{(1)}_{g}}{u^{(1)}_{g s_1}}{u^{(1)}_{g s_2}} = \edcedge{\alpha}{\gamma}{\gamma} \} = A \setminus \{\overline{g}\}$. If $A_1 = \emptyset$, then $\overline{u} = \overline{u}^{(1)}$. Otherwise, we can construct $u^{(2)}$ via the same argument. Repeat the procedure and construct $\{A_n\}$ and $\{u^{(n)}\}$ with $A_n \supsetneq A_{n+1}$ if $A_n \ne \emptyset$. Let $N = \norm{A_0}$ and $\overline{u} = u^{(N)}$ is the desired pattern.
		
Therefore, $\overline{u}$ can be found by construction, the proof is then complete.
\end{proof}

% ------------------------------------------------------
\section{Extended Directed Graph}

Given a finite set $V$, a \emph{directed graph} is a pair $(V, E)$ consists of \emph{vertex set} $V$ and \emph{edge set} $E \subset V \times V$. Directed graph plays an important role in the investigation of shifts of finite type. For instance, it is known that a shift of finite type in dimension one is irreducible if and only if its essential graph representation is strongly connected (cf.~\cite{LM-1995}), where a directed graph is called strongly connected if for any two vertices $v_1, v_2$ there is a path from $v_1$ to $v_2$. An \emph{extended directed graph} is an ordered triplet $\mathcal{G}=(V, E_c, E_d)$ defined as follows.
\begin{enumerate}
	\item $V$ is called the \emph{vertex set}.
	\item $E_c\subset V\times V$ is called the \emph{convergent-edge set}.
	\item $E_d\subset V\times V\times V$ is the \emph{divergent-edge set}.
\end{enumerate}
The ordered pair $\mathcal{G}^c = (V, E_c)$ of an extended directed graph $\mathcal{G}$ is called the \emph{intrinsic graph} of $\mathcal{G}$. Note that the divergent-edge set always consists of those edges $\edcedge{\alpha}{\beta}{\gamma}$ satisfying $\beta \neq \gamma$. It is seen that the intrinsic graph of an extended directed graph is a classical directed graph.

Suppose $\mathcal{G} = (V, E_c, E_d)$ is an extended directed graph. Define $\mathcal{F} \subset V^{\Delta_1}$ as
$$
\mathcal{F} = \{\edcedge{\alpha}{\beta}{\beta}: (\alpha, \beta) \notin E_c\} \bigcup \{\edcedge{\alpha}{\beta}{\gamma}: (\alpha, \beta, \gamma) \notin E_d, \beta \neq \gamma\}
$$
and let $\mathsf{X}_{\mathcal{G}} = \mathsf{X}_{\mathcal{F}}$. The following proposition extends a classical result in symbolic dynamics.

\begin{proposition}
Given an extended directed graph $\mathcal{G} = (V, E_c, E_d)$. Then $\mathsf{X}_{\mathcal{G}}$ is a TSFT over the alphabet $V$.
\end{proposition}

\begin{figure}
\begin{center}
\tikzset{
  r/.style = {draw, circle,fill = red!90!white, font=\bfseries, text =white},
  b/.style = {draw, circle,fill = blue!70!yellow, text = white},
  g/.style = {draw, circle,fill = green!70!yellow, text = white}
}
\tikzstyle{uni_arrow_edge} = [->, draw=black!90, line width=1]{}
\tikzstyle{div_arrow_edge} = [->, dashed, draw=black!90, line width=1]
\begin{tikzpicture}[
    scale = 0.8, transform shape, thick,
    every node/.style = {minimum size = 10mm},
    grow = down,  % alignment of characters
    level 1/.style = {sibling distance=2cm},
    level 2/.style = {sibling distance=3.5cm}, 
    level 3/.style = {sibling distance=1.5cm}, 
    level distance = 1.5cm,
  ]
\node[r] (a) {0} [grow'=down]
child {node[b] {1}}
child {node[b] {1}
};
\node[r] at (8,0) (a) {0} [grow'=down]
child {node[g] {2}}
child {node[b] {1}};

\node[r] (v1) at (-1,-4) {0};
\node[b] (v2) at (1,-4) {1};

\node[r] (v3) at (7,-4) {0};
\node[b] (v4) at (9,-3) {1};
\node[g] (v5) at (9,-5) {2};

\draw  (v1) edge[uni_arrow_edge] (v2);
\draw[div_arrow_edge] (7.1,-3.5) .. controls (7.6,-3.2) and (7.9,-3.1) .. (8.5,-3);
\draw[div_arrow_edge] (7.1,-4.5) .. controls (7.6,-4.8) and (7.9,-4.9) .. (8.5,-5);
\end{tikzpicture}
\end{center}
\caption{The extended directed graph representation of allowable blocks. An extended directed graph consists of a vertex set and two edge sets, convergent-edge set and divergent-edge set. Suppose $u = \edcedge{\alpha}{\beta}{\gamma}$ is an allowable block. Then $u$ is represented by a convergent-edge if $\beta = \gamma$, and $u$ is represented by a divergent-edge otherwise. We use solid line and dashed line to distinguish a convergent-edge from a divergent-edge.}
\label{fig:extend-graph}
\end{figure}
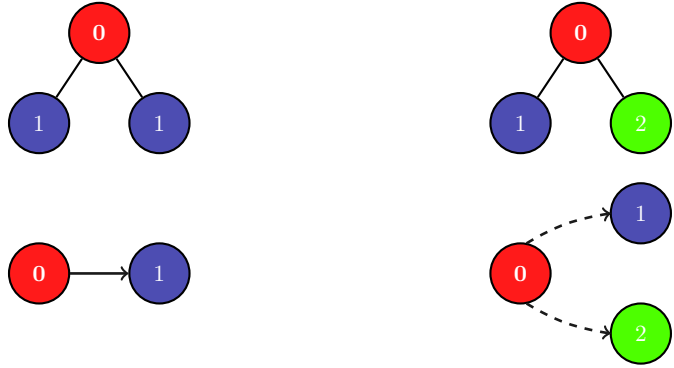

Now suppose $X$ is a TSFT induced by $B \subseteq \mathcal{A}^{\Delta_1}$. An \emph{extended directed graph representation} of $X$ is $\mathcal{G}=(V, E_c, E_d)$ defined as follows.
\begin{enumerate}[\bf 1)]
\item $V=\mathcal{A}$;
\item $E_c = \{(\alpha, \beta) \in \mathcal{A} \times \mathcal{A}: \edcedge{\alpha}{\beta}{\beta} \in B\}$;
\item $E_d = \{(\alpha, \beta, \gamma) \in \mathcal{A} \times \mathcal{A}  \times \mathcal{A}: \beta \ne \gamma, \edcedge{\alpha}{\beta}{\gamma} \in B\}$.
\end{enumerate}
With abuse of notation, we also denote $(\alpha, \beta) \in E_c$ and $(\alpha, \beta, \gamma) \in E_d$ as ${\edcedge{\alpha}{\beta}{\beta}} \in E_c$ and ${\edcedge{\alpha}{\beta}{\gamma}} \in E_d$, respectively. See Figure \ref{fig:extend-graph}. The following proposition, which is analogous to a result in classical symbolic dynamics, comes immediately.

\begin{proposition}
Suppose  $X$ is a TSFT and $\mathcal{G}$ is an extended directed graph representation of $X$. Then $\mathsf{X}_{\mathcal{G}} = X$.
\end{proposition}

Similar to the strong connectedness of a directed graph (also known as irreducible graph), we introduce the $(d, c)$-irreducibility of an extended directed graph as follows.

\begin{definition}\label{def:dc-red}
Let $\mathcal{G}=(V,E_c,E_d)$ be an extended directed graph.
\begin{enumerate}[\bf (1)]
\item $\mathcal{G}$ is called \emph{$(d,c)$-reducible} if there exists $\edcedge{\alpha}{\beta}{\gamma} \in E_d$ (respectively $\edcedge{\alpha}{\gamma}{\beta} \in E_d$) such that there exists a path in $\mathcal{G}^c$ from $\beta$ to $\gamma$ (respectively from $\gamma$ to $\beta$) and $\edcedge{\alpha}{\gamma}{\gamma} \notin E_c$ (respectively $\edcedge{\alpha}{\beta}{\beta} \notin E_c$).
\item  $\mathcal{H}:= (V, E_c \bigcup \{\edcedge{\alpha}{\gamma}{\gamma}\}, E_d)$ is called a \emph{$(d,c)$-reduction} of $\mathcal{G}$, denoted by $\mathcal{G} \preceq \mathcal{H}$.
\item $\mathcal{G}$ is called \emph{$(d,c)$-irreducible} if it is not $(d,c)$-reducible.
\end{enumerate}
\end{definition}

It can be seen that $\preceq$ defined in Definition \ref{def:dc-red} is a partial order on the set of extended directed graphs. In addition, such a partially ordered set is well-ordered. The $(d, c)$-reduction for each extended directed graph $\mathcal{G}$ is finite; that is, there exists $N \in \mathbb{N}$ such that $\mathcal{G}_n = \mathcal{G}_m$ for $n, m \geq N$, where $\mathcal{G} = \mathcal{G}_0 \preceq \mathcal{G}_1 \preceq \mathcal{G}_2 \preceq \cdots$.

\begin{remark}
Observe that the $(d, c)$-reduction is symmetric. More specifically, if there exists $\edcedge{\alpha}{\gamma}{\beta} \in E_d$ such that there is a path in $\mathcal{G}^c$ from $\beta$ to $\gamma$ and $\edcedge{\alpha}{\gamma}{\gamma} \notin E_c$, then we can still construct $\mathcal{H}$ as Definition \ref{def:dc-red} does.
\end{remark}

\begin{theorem} \label{thm:dc-reduction-keep-irr}
Let $X$ be a TSFT and let $\mathcal{G}=(\mathcal{A},E_c,E_d)$ be an extended directed graph representation of $X$. Suppose $\mathcal{H}$ is an extended directed graph such that $\mathcal{G} \preceq \mathcal{H}$. Then $\mathsf{X}_{\mathcal{G}}$ is CPC-irreducible if and only if $\mathsf{X}_{\mathcal{H}}$ is CPC-irreducible.
\end{theorem}
\begin{proof}
Suppose $\mathcal{H} = (\mathcal{A},E'_c,E_d)$ is a $(d,c)$-reduction of $\mathcal{G}$ such that $\edcedge{\alpha}{\gamma}{\gamma} \in E'_c \setminus E_c$. In other words, $\edcedge{\alpha}{\beta}{\gamma} \in E_d$ and  there exists a path $\delta_0 \delta_1 \ldots \delta_N$ in $\mathcal{G}^c$ such that $\delta_0 = \beta$ and $\delta_N = \gamma$ for some $\beta \in \mathcal{A}$. Define $u \in \mathcal{A}^{\Delta_N}$ as $u_g := \delta_{\norm{g}}$ for $g$ with $\norm{g} \leq N$. It follows from Proposition \ref{prop:extensible} that there exists an $x \in \mathsf{X}_{\mathcal{G}}$ such that $x|_{\Delta_N} = u$. Similar to the discussion in the proof of Theorem \ref{thm:replacement}, $\mathsf{X}_{\mathcal{G}}$ is CPC-irreducible if and only if $\mathsf{X}_{\mathcal{H}}$ is CPC-irreducible.
\end{proof}

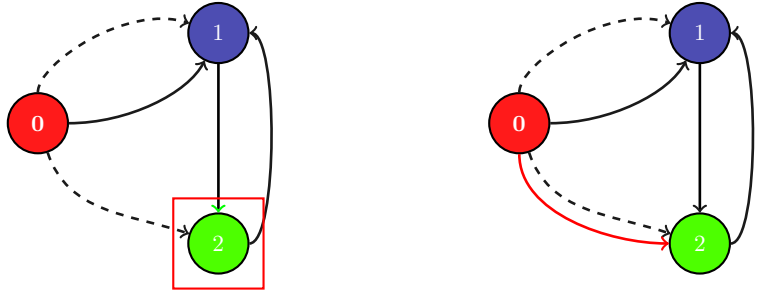
\begin{figure}
\tikzset{
  r/.style = {draw, circle,fill = red!90!white, font=\bfseries, text =white},
  b/.style = {draw, circle,fill = blue!70!yellow, text = white},
  g/.style = {draw, circle,fill = green!70!yellow, text = white}
}
\tikzstyle{uni_arrow_edge} = [->, draw=black!90, line width=1]{}
\tikzstyle{div_arrow_edge} = [->, dashed, draw=black!90, line width=1]
\begin{tikzpicture}[
    scale = 0.8, transform shape, thick,
    every node/.style = {minimum size = 10mm},
    grow = down,  % alignment of characters
    level 1/.style = {sibling distance=2cm},
    level 2/.style = {sibling distance=3.5cm}, 
    level 3/.style = {sibling distance=1.5cm}, 
    level distance = 1.5cm,
  ]

\node[r] at (-2,0) (v3) {0};
\node[b] (v4) at (1,1.5) {1};
\node[g] (v5) at (1,-2) {2};
\draw[div_arrow_edge] (v3) .. controls (-2,1) and (-0.5,2) .. (v4);
\draw[uni_arrow_edge] (v3) .. controls (-0.5,0) and (0.5,0.5) .. (v4);

\draw[div_arrow_edge] (v3) .. controls (-1.5,-1.5) and (-0.5,-1.5) .. (v5);

\draw[uni_arrow_edge,draw=green]  (v4) edge (v5);
\draw[uni_arrow_edge] (v5) .. controls (2,-2) and (2,1.5) .. (v4);

\node[r] at (6,0) (v6) {0};
\node[b] (v7) at (9,1.5) {1};
\node[g] (v8) at (9,-2) {2};
\draw[div_arrow_edge] (v6) .. controls (6,1) and (7.5,2) .. (v7);
\draw[uni_arrow_edge] (v6) .. controls (7.5,0) and (8.5,0.5) .. (v7);

\draw[div_arrow_edge] (v6) .. controls (6.5,-1.5) and (7.5,-1.5) .. (v8);
\draw[uni_arrow_edge,draw=red] (v6) .. controls (6,-1.5) and (7.5,-2) .. (v8);

\draw[uni_arrow_edge]  (v7) edge (v8);
\draw[uni_arrow_edge] (v8) .. controls (10,-2) and (10,1.5) .. (v7);

\draw[draw=red]  (0.25,-1.25) rectangle (1.75,-2.75);
\end{tikzpicture}
\caption{The graph on the right-hand side is a $(d,c)$-reduction of the graph on the left-hand side, and is $(d, c)$-irreducible.}
\label{fig:dc-reduction}
\end{figure}

\begin{example} \label{eg:dcreduction}
Suppose $\mathcal{G} = (V, E_c, E_d)$ is an extended directed graph given as $V = \{0, 1, 2\}, E_c = \{(0, 1), (1, 2), (2, 1)\}$, and $E_d = \{(0, 1, 2)\}$. It is seen that $\mathcal{G}$ is $(d, c)$-reducible. The $(d, c)$-reduction of $\mathcal{G}$ is $\mathcal{H} = (V, E_c', E_d)$ with $E_c' = E_c \bigcup \{(0, 2)\}$. Furthermore, $\mathcal{H}$ is $(d, c)$-irreducible. See Figure \ref{fig:dc-reduction}.
\end{example}

% ------------------------------------------------------
\section{Decidability of CPC-Irreducibility}

The preceding theorem raises the question of relationship between CPC-irreducible TSFTs and $(d, c)$-irreducible extended directed graphs. A classical result in symbolic dynamics is that a ($1$-step) shift of finite type is irreducible if and only if it has an irreducible directed graph representation. In addition, every shift of finite type is topologically conjugate with a $1$-step shift induced by some directed graph. This transfers the discussion of irreducible SFTs to irreducible graphs. Since $\mathcal{G}$ is finite, there exists $N \in \mathbb{N}$ such that $\mathcal{G}_n = \mathcal{G}_m$ for $n, m \geq N$, where $\mathcal{G}_n \preceq \mathcal{G}_{n+1}$ and $\mathcal{G}_1 = \mathcal{G}$.

\begin{theorem} \label{thm:cpc-irr-extgraph-strong-connect}
Let $X$ be a TSFT and let $\mathcal{G}=(\mathcal{A},E_c,E_d)$ be a $(d,c)$-irreducible extended directed graph representation of $X$. Then, $\mathsf{X}_{\mathcal{G}}$ is CPC-irreducible if and only if $\mathcal{G}^c$ is strongly connected.
\end{theorem}

\begin{figure}
\begin{center}
\tikzset{
  r/.style = {>=stealth,draw, circle,fill = red!90!white, font=\bfseries, text =white},
  b/.style = {>=stealth,draw, circle,fill = blue!70!yellow, text = white},
  g/.style = {>=stealth,draw, circle,fill = green!70!yellow, text = white}
}
\tikzstyle{uni_arrow_edge} = [->, draw=black!90, line width=1]{}
\tikzstyle{div_arrow_edge} = [->, dashed, draw=black!90, line width=1]
\begin{tikzpicture}[
    scale = 0.6, transform shape, thick,
    every node/.style = {minimum size = 10mm},
    grow = down,  % alignment of characters
    level 1/.style = {sibling distance=2cm},
    level 2/.style = {sibling distance=3.5cm}, 
    level 3/.style = {sibling distance=1.5cm}, 
    level distance = 1.5cm,
  ]

\node at (11,0) {\Large $\mathcal{G}$};
\node[r,fill=cyan] at (-2,0) (v3) {0};
\node[r,fill=pink] (v4) at (1,1.5) {1};
\node[r] (v5) at (1,-2) {2};

\draw[div_arrow_edge] (v3) .. controls (-2,1) and (-0.5,2) .. (v4);
\draw[uni_arrow_edge] (v3) .. controls (-0.5,0) and (0.5,0.5) .. (v4);

\draw[div_arrow_edge] (v3) .. controls (-1.5,-1.5) and (-0.5,-1.5) .. (v5);

\draw[uni_arrow_edge,draw=black]  (v4) edge (v5);
\draw[uni_arrow_edge] (v5) .. controls (0,-1) and (-0.5,-0.5) .. (v3);

\node[g,fill=yellow] (v7) at (9,1.5) {3};
\node[g] (v8) at (9,-2) {4};

\draw[uni_arrow_edge]  (v7) edge (v8);
\draw[uni_arrow_edge] (v8) .. controls (10,-2) and (10,1.5) .. (v7);

\node[b] (v9) at (5,5) {5};

\draw[div_arrow_edge] (v9) .. controls (3.5,5.5) and (1,3) .. (v4);
\draw[div_arrow_edge] (v9) .. controls (5,4) and (3,-2) .. (v5);
\draw[uni_arrow_edge] (v5) .. controls (1.5,-3.5) and (8.5,-3.5) .. (v8);
\draw[uni_arrow_edge] (v7) .. controls (9,2.5) and (6,5.5) .. (v9);
% \draw[draw=red]  (-3,2.5) rectangle (3,-3.5);
% \draw[draw=red]  (4,6) rectangle (6,4);
% \draw[draw=red]  (7,3) rectangle (11,-4);

\end{tikzpicture}
\end{center}

\begin{center}
\tikzset{
  r/.style = {>=stealth,draw, circle,fill = red!90!white, font=\bfseries, text =white},
  b/.style = {>=stealth,draw, circle,fill = blue!70!yellow, text = white},
  g/.style = {>=stealth,draw, circle,fill = green!70!yellow, text = white}
}
\tikzstyle{uni_arrow_edge} = [->, draw=black!90, line width=1]{}
\tikzstyle{div_arrow_edge} = [->, dashed, draw=black!90, line width=1]
\begin{tikzpicture}[
    scale = 0.6, transform shape, thick,
    every node/.style = {minimum size = 10mm},
    grow = down,  % alignment of characters
    level 1/.style = {sibling distance=2cm},
    level 2/.style = {sibling distance=3.5cm}, 
    level 3/.style = {sibling distance=1.5cm}, 
    level distance = 1.5cm,
  ]

\node at (11,0) {\Large $\mathcal{H}$};
\node[r,fill=cyan] at (-2,0) (v3) {0};
\node[r,fill=pink] (v4) at (1,1.5) {1};
\node[r] (v5) at (1,-2) {2};

\draw[div_arrow_edge] (v3) .. controls (-2,1) and (-0.5,2) .. (v4);
\draw[uni_arrow_edge] (v3) .. controls (-0.5,0) and (0.5,0.5) .. (v4);

\draw[div_arrow_edge] (v3) .. controls (-1.5,-1) and (-0.5,-1.5) .. (v5);

\draw[uni_arrow_edge,draw=black]  (v4) edge (v5);
\draw[uni_arrow_edge] (v5) .. controls (0,-1) and (-0.5,-0.5) .. (v3);
\draw[uni_arrow_edge,draw=red] (v3) .. controls (-1.5,-1.5) and (-0.5,-2.5) .. (v5);

\node[g,fill=yellow] (v7) at (9,1.5) {3};
\node[g] (v8) at (9,-2) {4};

\draw[uni_arrow_edge]  (v7) edge (v8);
\draw[uni_arrow_edge] (v8) .. controls (10,-2) and (10,1.5) .. (v7);

\node[b] (v9) at (5,5) {5};

\draw[div_arrow_edge] (v9) .. controls (3.5,5.5) and (1,3) .. (v4);
\draw[div_arrow_edge] (v9) .. controls (5,4) and (3,-2) .. (v5);
\draw[uni_arrow_edge,draw=red] (v9) edge (v5);
\draw[uni_arrow_edge] (v5) .. controls (1.5,-3.5) and (8.5,-3.5) .. (v8);
\draw[uni_arrow_edge] (v7) .. controls (9,2.5) and (6,5.5) .. (v9);
% \draw[draw=red]  (-3,2.5) rectangle (3,-3.5);
% \draw[draw=red]  (4,6) rectangle (6,4);
% \draw[draw=red]  (7,3) rectangle (11,-4);

\end{tikzpicture}
\end{center}
\caption{Extended directed graph representation $\mathcal{G}$ of TSFT in Example \ref{eg:dc-reduction-strong-connect} (the above one). The graph $\mathcal{H}$ is the $(d, c)$-reduction of $\mathcal{G}$ and is $(d, c)$-irreducible (the two red edges are generated convergent-edges). In addition, it is easily seen that $\mathcal{H}^c$ is strongly connected.}
\label{fig:dc-reduction-strong-connect}
\end{figure}
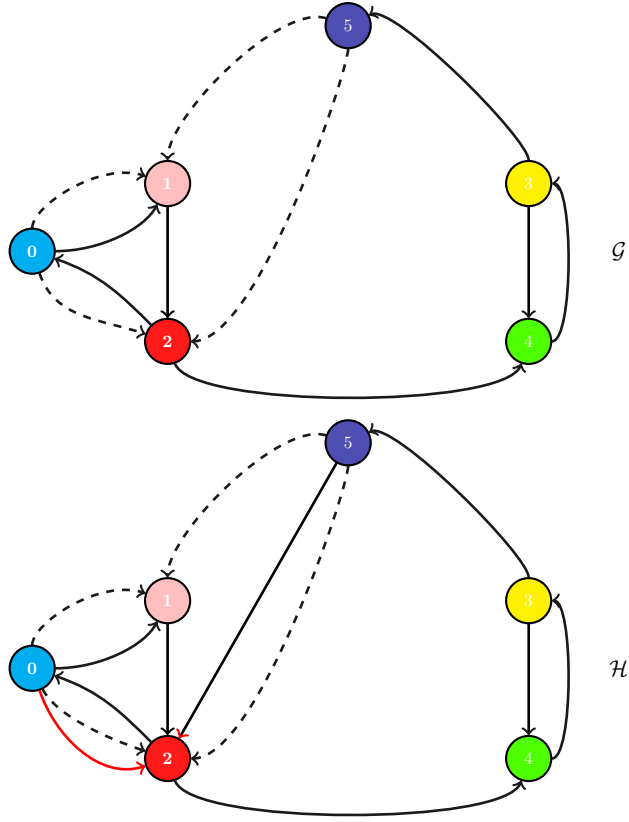

\begin{example} \label{eg:dc-reduction-strong-connect}
Suppose $\mathcal{A} = \{0, 1, 2, 3, 4, 5\}$ and $X$ is a TSFT induced by
$$
B = \left\{
\begin{aligned}
&\edcedge{0}{1}{1}, \edcedge{0}{1}{2}, \edcedge{1}{2}{2}, \edcedge{2}{0}{0}, \edcedge{2}{4}{4}, \\
  &\edcedge{3}{4}{4}, \edcedge{3}{5}{5}, \edcedge{4}{3}{3}, \edcedge{5}{1}{2}
\end{aligned}\right\}.
$$
The graph $\mathcal{G}$ above in Figure \ref{fig:dc-reduction-strong-connect} is an extended directed graph representation of $X$. Observe that $\mathcal{G}$ is $(d, c)$-reducible and there are two convergent-edges ($(0, 2)$ and $(5, 2)$) generated via $(d, c)$-reduction. Let $\mathcal{H}$ be the $(d, c)$-irreducible extended directed graph of $\mathcal{G}$. It follows immediately that $\mathcal{H}^c$ is strongly connected. Theorem \ref{thm:cpc-irr-extgraph-strong-connect} demonstrates that $X$ is CPC-irreducible.
\end{example}

\begin{lemma} \label{lem:irr_cycle}
Let $X$ be a CPC-irreducible TSFT and let $\mathcal{G}$ be an extended directed graph representation of $X$. Suppose $V = V_1 \bigcup V_2 \bigcup \ldots \bigcup V_N$ such that $N>=2$ and $V_1, . . . , V_N$ are strongly connected components. Then
\begin{enumerate}
\item for each $V_i$, there exists $\beta, \gamma \in V_i$, $\alpha \in V_j$ ($i \ne j$) such that $\edcedge{\alpha}{\beta}{\gamma} \in E_c \bigcup E_d$ (denoted by $V_j \xrightarrow[(\beta,\gamma)]{\alpha} V_i$);
\item there exist distinct $V_{i_1}, \ldots, V_{i_M}$, where $1 \le i_1 \le \cdots \le i_M \le N$, and $\alpha_{i_{j}}, \beta_{i_{j}}, \gamma_{i_{j}} \in V_{i_{j}}$ such that
$$
V_{i_1} \xrightarrow[(\beta_{i_{M}}, \gamma_{i_M})]{\alpha_{i_{1}}} V_{i_M} \xrightarrow[(\beta_{i_{M-1}},\gamma_{i_{M-1}})]{\alpha_{i_{M}}}\ldots \xrightarrow[(\beta_{i_2}, \gamma_{i_2})]{\alpha_{i_3}} V_{i_2} \xrightarrow[(\beta_{i_1}, \gamma_{i_1})]{\alpha_{i_2}} V_{i_1}.
$$
In addition, there exists $j_0 \in \{i_1, \ldots, i_M\}$ such that ${\beta_{j_0} \ne \gamma_{j_0}}$, $\edcedge{\alpha_{j_0 + 1}}{\beta_{j_0}}{\beta_{j_0}} \notin E_c$ and $\edcedge{\alpha_{j_0 + 1}}{\gamma_{j_0}}{\gamma_{j_0}} \notin E_c$. 
\end{enumerate}
\end{lemma}
\begin{proof}
(1) Suppose not. Then, there exists $V_i$ such that for every $\beta, \gamma \in V_i$, and $\alpha \in V \setminus V_i$, $\edcedge{\alpha}{\beta}{\gamma} \notin E_c \bigcup E_d$. However, for every given $x$ with $x_\epsilon \in V \setminus V_i$, there exists a sequence $\{w_n\}_{n=0}^\infty \subset \Sigma^*$ such that $\norm{w_n}=n$, $w_n$ is a prefix of $w_{n+1}$ and $x_{w_n} \in V \setminus V_i$. This can be proved by the existence of $\{w_n\}_{n=0}^N$ such that $\norm{w_n} = n$, $w_n$ is a prefix of $w_{n+1}$ and $x_{w_n} \in V \setminus V_i$ by induction on $N$.
		
When $n = 0$, $x_{w_0} \in V \setminus V_i$ by the definition of $x$. Suppose the assertion holds for $N$, then we verify the case $N+1$. Since $\edcedge{x_{w_N}}{x_{w_{N} s_1}}{x_{w_{N} s_2}} \in B$ and $x_{w_n} \in V_k$ for some $V_k \subseteq V \setminus V_i$, thus for some $a \in \{s_1, s_2\}$, $x_{w_N a} \in V \setminus V_i$. Let $w_{N+1}:=w_N a$ and the induction hypothesis holds. 
		
By mathematical induction, the assertion holds for all $N \ge 0$. Since $X$ is CPC-irreducible, given any $\hat{\alpha} \in V \setminus V_i$ and $\hat{\beta} \in V_i$ there exist $y \in X$ and CPC $S$ such that $y_\epsilon = \hat{\alpha}$ and $y_g = \hat{\beta}$ for all $g \in S$ by Lemma \ref{lem:CPC_irr}. However, the argument above guarantees the existence of a particular $w_{N_0} \in \{w_n\}_{n=0}^{\infty} \cap S$ that $\hat{\beta} = y_{w_{N_0}} \in V \setminus V_i$, which is a contradiction.
		
(2) First we prove the existence of $V_{i}$'s by contradiction. Suppose not, we may find distinct $V_{i_1}, V_{i_2}, \ldots, V_{i_{N+1}}$ along with $\alpha_{i_j}, \beta_{i_j}, \gamma_{i_j} \in V_{i_j}$ such that
$$
{V_{i_{N+1}} \xrightarrow[(\beta_{i_{N}}, \gamma_{i_N})]{\alpha_{i_{N+1}}} V_{i_N} \xrightarrow[(\beta_{i_{N-1}},\gamma_{i_{N-1}})]{\alpha_{i_{N}}}\ldots \xrightarrow[(\beta_{i_2}, \gamma_{i_2})]{\alpha_{i_3}} V_{i_2} \xrightarrow[(\beta_{i_1}, \gamma_{i_1})]{\alpha_{i_2}} V_{i_1}},
$$
which is impossible since there are only $N$ distinct $V_i$'s. 
        
Now we show that the existence of $j_0$ naturally follows from the existence of $\{V_{i_k}\}_{k=1}^{M}$ by contradiction. For simplicity, denote $i_{M+1} = i_{1}$. Suppose for each $1 \le j \le M$, there exists some $\delta_{i_{j+1}} \in \{\beta_{i_{j+1}}, \gamma_{i_{j+1}}\}$ such that $\edcedge{\alpha_{i_j}}{\delta_{i_{j+1}}}{\delta_{i_{j+1}}} \in E_c$. Then it contradicts that $\mathcal{G}^c|_{\bigcup_{k=1}^{M} V_{i_k}}$ is not strongly connected. It is because for all $\alpha \in V_{i_{j_0}}, \beta \in V_{i_{j_1}}$, if $j_0 < j_1$, $\alpha, \alpha_{i_{j_0}}, \delta_{i_{j_0 +1}}, \alpha_{i_{j_0 +1}}, \delta_{i_{j_0 +2}}, \ldots, \delta_{i_{j_1}}, \beta$ is a sequence in which neighboring vertex are connected. Such a sequence also exists for the cases $j_0 > j_1$ and $j_0 = j_1$ respectively. Hence, $\mathcal{G}^c|_{\bigcup_{j=1}^{M} V_{i_j}}$ is strongly connected, which contradicts the assumption.
\end{proof}

\begin{proof}[Proof of Theorem \ref{thm:cpc-irr-extgraph-strong-connect}]
For sufficient condition, given any $\alpha, \beta \in V$, there exists path $\gamma_0 \gamma_1 \ldots \gamma_N \gamma_{N+1}$ in $\mathcal{G}^c$ such that $\gamma_0 = \alpha$ and $\gamma_{N+1} = \beta$ since $\mathcal{G}^c$ is strongly connected. Hence, $\{\edcedge{\gamma_i}{\gamma_{i+1}}{\gamma_{i+1}}\}_{i=0}^{N} \subseteq E_c$ and thus $\{\edcedge{\gamma_i}{\gamma_{i+1}}{\gamma_{i+1}}\}_{i=0}^{N} \subseteq B$, where $B \subseteq \mathcal{A}^{\Delta_1}$ is the set of allowable one-blocks. Therefore, for each $\alpha, \beta \in \mathcal{A}$, there exists some $(N+1)$-block $u^{\ket{\alpha, \beta}}$, where $u^{\ket{\alpha, \beta}}_g:=\gamma_{\norm{g}}$ for $\norm{g} \le N+1$. In this case, $\partial s(u^{\ket{\alpha, \beta}})$ is a CPC and for each $g \in s(u^{\ket{\alpha, \beta}})$ with $g \Delta_1 \subseteq s(u^{\ket{\alpha, \beta}})$, $\edcedge{u^{\ket{\alpha, \beta}}_g}{u^{\ket{\alpha, \beta}}_{g s_1}}{u^{\ket{\alpha, \beta}}_{g s_2}} = \edcedge{\gamma_{\norm{g}}}{\gamma_{\norm{g}+1}}{\gamma_{\norm{g}+1}} \in B$. By Proposition \ref{prop:extensible}, there exists $x^{\ket{\alpha,\beta}} \in \mathsf{X}_{\mathcal{G}}$ such that $x|_{s(u)} = u$. Since $\alpha, \beta$ are arbitrary, by Lemma \ref{lem:CPC_irr}, $\mathsf{X}_\mathcal{G}$ is CPC irreducible.
		
For necessary condition, we prove the statement by contradiction. Suppose $\mathcal{G}^c$ is not strongly connected, then by Lemma \ref{lem:irr_cycle}, $V = V_1 \bigcup V_2 \ldots \bigcup V_M$, where $M>1$ and there exist some $V_i, V_{i+1}$ such that $V_i \xrightarrow[(\beta_{i+1},\gamma_{i+1})]{\alpha_i} V_{i+1}$ with $\beta_{i+1} \ne \gamma_{i+1}$, $\edcedge{\alpha_{i}}{\beta_{i+1}}{\beta_{i+1}} \notin E_c$ and $\edcedge{\alpha_{i}}{\gamma_{i+1}}{\gamma_{i+1}} \notin E_c$. It is seen that $\mathcal{G}$ is $(d,c)$-reducible, which is a contradiction.
\end{proof}

Suppose that $\mathcal{G} = (V, E_c, E_d)$ is $(d, c)$-reducible such that $c:=\edcedge{\alpha}{\beta}{\beta} \notin E_c$, $d:=\edcedge{\alpha}{\beta}{\gamma} \in E_d$, and $\beta \beta_1 \beta_2 \ldots \beta_k \gamma$ is a path in $\mathcal{G}^c$. We say that $\mathcal{H}:=(V, E_c \bigcup \{c\}, E_d \setminus \{d\})$ is an \emph{enhanced $(d,c)$-reduction} of $\mathcal{G}$. It is seen that $\mathsf{X}_{\mathcal{G}}$ is CPC-irreducible if and only if $\mathsf{X}_{\mathcal{H}}$ is CPC-irreducible.

\begin{theorem}
Suppose $\mathcal{G} = (V, E_c, E_d)$ is an extended directed graph and $\mathcal{H}$ is an enhanced $(d,c)$-reduction of $\mathcal{G}$. Then $\mathsf{X}_{\mathcal{G}}$ is CPC-irreducible if and only if $\mathsf{X}_{\mathcal{H}}$ is CPC-irreducible.
\end{theorem}
\begin{proof}
The proof is similar to the discussion of Theorem \ref{thm:dc-reduction-keep-irr}, thus it is omitted.
\end{proof}

We conclude this section by improving Theorem \ref{thm:cpc-irr-extgraph-strong-connect} by considering irreducible components of the extended directed graph. Let $\mathcal{G}=(V,E_c,E_d)$ be an extended directed graph. Decompose $V = V_1 \bigcup V_2 \bigcup \cdots \bigcup V_n$ so that $\mathcal{G}^c|_{V_i}$ is the biggest strongly connected component for eah $1 \leq i \leq N$ as Lemma \ref{lem:CPC_irr} did. Define $\mathcal{H}=(\widetilde{V},\widetilde{E_c} ,\widetilde{E_d})$ as follows.
\begin{itemize}
\item $\widetilde{V}=\left\{V_1,\ldots,V_N\right\}$.
\item $\widetilde{E_c}=\left\{(V_i, V_j): \exists \alpha\in V_i, \beta, \gamma \in V_j, \edcedge{\alpha}{\beta}{\gamma} \in E_c \bigcup E_d\right\}$.
\item $\widetilde{E_d}=\left\{(V_i, V_j, V_k): j \ne k, \exists \alpha \in V_i, \beta \in V_j, \gamma \in V_k, \edcedge{\alpha}{\beta}{\gamma} \in E_d \right\}$.
\end{itemize}
Then we call $\mathcal{H}=(\widetilde{V},\widetilde{E_c} ,\widetilde{E_d})$ a \emph{grouping $(d, c)$-reduction} of $\mathcal{G}$. Observe that grouping $(d, c)$-reduction of an extended directed graph is unique up to permutation. Furthermore, we refer to $\overline{\mathcal{G}}$ as the limit of $(d, c)$-reduction of $\mathcal{G}$ which is $(d,c)$-irreducible.

\begin{theorem} \label{thm:group-dcreduction-CPCirr}
Let $\mathcal{G}$ be an extended directed graph and let $\mathcal{H}$ be the grouping $(d,c)$-reduction of $\mathcal{G}$. Then $\overline{\mathcal{G}}^c$ is strongly connected if and only if $\overline{\mathcal{H}}^c$ is strongly connected.
\end{theorem}
\begin{proof}
		Define $\mathcal{G}_0 := \mathcal{G}$ and $\mathcal{H}_0 := \mathcal{H}$, where $\widetilde{V} = \left\{ V_j \right\}_{j=1}^{M_{\mathcal{G}}}$, $\mathcal{G}_0 = (V,E_c^{(0)},E_d)$ and $\mathcal{H}_0 = (\widetilde{V},\widetilde{E_c}^{(0)},\widetilde{E_d})$.

Let $\mathcal{G} \preceq \mathcal{G}_1 \preceq \mathcal{G}_2 \ldots \preceq \mathcal{G}_N = \overline{\mathcal{G}}$, where $\mathcal{G}_{i+1}$ is the $(a_i \to (b_i,c_i), a_i \to (c_i,c_i))$-reduction of $\mathcal{G}_{i}$ with $a_i \in \tilde{V}_{m_i}, b_i \in \tilde{V}_{n_i}, c_i \in \tilde{V}_{k_i}$, as is defined in Definition \ref{def:dc-red}. Herein, $(a_i \to (b_i,c_i), a_i \to (c_i,c_i))$-reduction of $\mathcal{G}_{i}$ means that the $(d, c)$-reduction of $\mathcal{G}_i$ is achieved by the divergent-edge $a_i \to (b_i,c_i)$ and the convergent-edge $a_i \to (c_i,c_i)$. Notably, $\overline{\mathcal{G}}$ is well-defined since for every $(\alpha_1, \alpha_2) \in V \times V$ (respectively $(\alpha_1, \alpha_2, \alpha_3) \in V \times V \times V$) there is at most one convergent-edge from $\alpha_1$ to $\alpha_2$ (respectively divergent-edge from $\alpha_1$ to $(\alpha_2, \alpha_3)$), this makes the operation of $(d, c)$-reduction stop in finitely many steps and leads to the same graph. Furthermore, if we are able to add an edge in the sequence $\mathcal{G}_i$, we will be able to add it in the other sequence $\mathcal{G}'_i$, where $\mathcal{G}_0 = \mathcal{G}'_0 = \mathcal{G}$. In other words, the procedure is ``confluent''.
%Then, following notations are used consistently throughout the proof:\\
%        $c_i = \edcedge{\alpha_{i}}{\gamma_{i}}{\gamma_{i}} \notin E_c^{(i)}$ and $d_i = \edcedge{\alpha_{i}}{\beta_{i}}{\gamma_{i}} \in E_d$ (or $d_i:=\edcedge{\alpha_{i}}{\gamma_{i}}{\beta_{i}} \in E_d$) for $0 \le i < N$, where $\alpha_i \in V_{m_i}, \beta_i \in V_{n_i}, \gamma_i \in V_{k_i}$ for some $V_{m_i}, V_{n_i}, V_{k_i} \in \widetilde{V}$. 
        
		Define $\mathcal{H}_{i+1} := \begin{cases}
		\mathcal{H}_{i}, & \text{if } V_{n_i} = V_{k_i}; \\
		(\widetilde{V}, \widetilde{E_c}^{(i)} \bigcup \{\edcedge{\widetilde{V}_{m_i}}{\widetilde{V}_{k_i}}{\widetilde{V}_{k_i}}\},\widetilde{E_d}), & \text{if } V_{n_i} \ne V_{k_i}. \\
		\end{cases}$
		
		First, consider a mapping $f_i: E_c^{(i)} \bigcup E_d \rightarrow (P(V) \times P(V) \times P(V)) \bigcup (P(V) \times P(V))$ defined as follows:
		\[f_i(\edcedge{\alpha}{\beta}{\gamma}) = \begin{cases}
		\edcedge{V_{m}}{V_{n}}{V_{n}}, & \text{if } \alpha \in V_m, \beta,\gamma \in V_{n}; \\
		\edcedge{V_{m}}{V_{n}}{V_{k}}, & \text{if } \alpha \in V_m, \beta \in V_{n}, \gamma \in V_k, V_n \ne V_k. \\
		\end{cases}\]
		Then, it can be shown that $f_i$ is a surjection from $E_c^{(i)} \bigcup E_d$ to $\widetilde{E_c}^{(i)} \bigcup \widetilde{E_d}$ by induction on $i$ for $0 \le i \le N$.
		
		When $i=0$, the assertion holds automatically by definition of grouping reduction.
		
		Suppose it holds for some $i \geq 0$. That is,
        \begin{itemize}
        	\item for each $\edcedge{\alpha}{\beta}{\gamma} \in E_c^{(i)} \bigcup E_d$, there exist some $1 \le j_1, j_2, j_3 \le M_{\mathcal{G}}$ and $\alpha \in V_{j_1}, \beta \in V_{j_2}, \gamma \in V_{j_3}$ such that $\edcedge{V_{j_1}}{V_{j_2}}{V_{j_3}} \in \widetilde{E_c}^{(i)} \bigcup \widetilde{E_d}$; 
            \item for each $\edcedge{V_{j_1}}{V_{j_2}}{V_{j_3}} \in \widetilde{E_c}^{(i)} \bigcup \widetilde{E_d}$, there exist $\alpha \in V_{j_1}, \beta \in V_{j_2}, \gamma \in V_{j_3}$ such that $\edcedge{\alpha}{\beta}{\gamma} \in E_c^{(i)} \bigcup E_d$. 
        \end{itemize}
		It is left to verify the case $i+1$. Note that $E_c^{(i+1)} = E_c^{(i)} \bigcup \{\edcedge{\alpha_i}{\gamma_i}{\gamma_i}\}$, where $\alpha_i \in \tilde{V}_{m_i}, \gamma_i \in \tilde{V}_{k_i}$. If $V_{n_i} \ne V_{k_i}$, $\widetilde{E_c}^{(i+1)} = \widetilde{E_c}^{(i)} \bigcup \{\edcedge{V_{m_i}}{V_{k_i}}{V_{k_i}}\}$. Thus, $f_i(\edcedge{\alpha_i}{\gamma_i}{\gamma_i}) = \edcedge{V_{m_i}}{V_{k_i}}{V_{k_i}} \in \widetilde{E_c}^{(i)}$. If $\beta_i = \gamma_i$, then $V_{n_i} = V_{k_i}$, $\widetilde{E_c}^{(i+1)} = \widetilde{E_c}^{(i)}$, and $f_i(\edcedge{\alpha_i}{\gamma_i}{\gamma_i}) = \edcedge{V_{m_i}}{V_{k_i}}{V_{k_i}} \in \widetilde{E_c}^{(i)}$ already. Hence the assumption holds for $i+1$. 
        
        By mathematical induction, the assertion holds for all $i$ with $1 \le i \le N$.
		
		Now, we would like to generalize the surjective mapping from edge to path by considering the following two assertions.
        
        Firstly, we prove that for each path $\delta_0 \delta_1 \ldots \delta_{M+1}$ in $\mathcal{G}_i^{c}$, there is a path $\eta_0 \eta_1 \ldots \eta_{M+1}$ in $\mathcal{H}_i^c$ such that $\delta_0 \in \eta_0$ and $\delta_{M+1} \in \eta_{M+1}$. It can be proved by considering $\edcedge{\eta_l}{\eta_{l+1}}{\eta_{l+1}}:=f_i(\edcedge{\delta_l}{\delta_{l+1}}{\delta_{l+1}})$ for $0 \le l \le M$ in the sense that $\eta_l \to (\eta_l, \eta_l)$ is defined to be connected.
		
		Secondly, if $\eta_0 \eta_1 \ldots \eta_{M+1}$ is a path in $\mathcal{H}_{N}^c$, without assuming any strong connectivity of $\mathcal{G}_N^c$ and $\mathcal{H}_N^c$, there exists a sequence $\delta_0, \delta_1, \ldots , \delta_{M+1}$, where $\delta_l \in \eta_l$ for all $0 \le l \le M+1$ and $\delta_l \nu_{l,1} \nu_{l,2} \ldots \nu_{l,K_{l}} \delta_{l+1}$ is a path in $\mathcal{G}_N^c$. It can be shown by the following argument:
		
		For each $l$, there exists $\edcedge{\delta_l}{\hat{\delta}_{l+1}}{\check{\delta}_{l+1}} \in E_c^{(N)} \bigcup E_d$ such that $f_N(\edcedge{\delta_l}{\hat{\delta}_{l+1}}{\check{\delta}_{l+1}})=\edcedge{\eta_l}{\eta_{l+1}}{\eta_{l+1}}$. In this case, $\edcedge{\delta_l}{\hat{\delta}_{l+1}}{\hat{\delta}_{l+1}} \in E_c^{(N)}$ since $\mathcal{G}_N^c|_{\eta_{l+1}}$ is strongly connected or $\eta_{l+1}$ consists of a single vertex, and $\delta_{l}, \hat{\delta}_{l}, \check{\delta}_{l} \in \eta_{l}$. Thus, for each $0 \le l \le M$ in $\mathcal{G}$, there exists a path $\delta_l \hat{\delta}_{l+1} \nu_{l+1,1} \nu_{l+1,2} \ldots \nu_{l+1,K_{l+1}} \delta_{l+1}$ as is required from above.
		
		From argument above, $\mathcal{H}_{i+1}$ is either a $(\edcedge{V_{m_i}}{V_{n_i}}{V_{k_i}},\edcedge{V_{m_i}}{V_{k_i}}{V_{k_i}})$-reduction of $\mathcal{H}_i$ or identical to $\mathcal{H}_i$. Herein, $(\edcedge{V_{m_i}}{V_{n_i}}{V_{k_i}},\edcedge{V_{m_i}}{V_{k_i}}{V_{k_i}})$-reduction means that the grouping $(d, c)$-reduction of $\mathcal{H}_i$ is accomplished by the divergent-edge $\edcedge{V_{m_i}}{V_{n_i}}{V_{k_i}}$ and the convergent-edge $\edcedge{V_{m_i}}{V_{k_i}}{V_{k_i}}$. It is because for $\edcedge{\alpha_i}{\beta_i}{\gamma_i}$ with $V_{n_i} \ne V_{k_i}$, there exists a path $\delta_0 \delta_1 \ldots \delta_{M+1}$ in $\mathcal{H}_i^c$, where $\delta_0 = \beta_i$, $\delta_{M+1} = \gamma_i$. Let $\edcedge{\eta_l}{\eta_{l+1}}{\eta_{l+1}} = f_i(\edcedge{\delta_l}{\delta_{l+1}}{\delta_{l+1}}) \in \widetilde{E_c}^{(i)}$, where $0 \le i \le {M}$. Then, $\eta_0 \eta_1 \ldots \eta_{M+1}$ is a path in $\mathcal{H}_i^c$, where $\eta_0 = V_{n_i}, \eta_{M+1} = V_{k_i}$. Hence, the assertion is satisfied.
		
		At this point, it can be shown by contradiction that $\mathcal{H}_N = \overline{\mathcal{H}}$, i.e.\ $\mathcal{H}_N$ is not $(d,c)$-reducible. If $\mathcal{H}_N \ne \overline{\mathcal{H}}$, there exists some $\edcedge{V_{m_{N+1}}}{{V_{n_{N+1}}}}{{V_{k_{N+1}}}} \in \widetilde{E_d}$ (or $\edcedge{V_{m_{N+1}}}{{V_{k_{N+1}}}}{{V_{n_{N+1}}}} \in \widetilde{E_d}$) such that $\eta_0 \eta_1 \ldots \eta_{M+1}$ is a path in $\mathcal{H}_N^c$, $\eta_0 = {V_{n_{N+1}}}$, $\eta_{M+1} = {V_{k_{N+1}}}$ and $\edcedge{V_{m_{N+1}}}{{V_{k_{N+1}}}}{{V_{k_{N+1}}}} \notin \widetilde{E_c}^{(N)}$. Since $f_N: E_c^{(N)} \bigcup E_d \rightarrow \widetilde{E_c}^{(N)} \bigcup \widetilde{E_d}$ is a surjection, there exist ${\alpha_{N+1}}, {\beta_{N+1}}, {\gamma_{N+1}} \in V$ satisfying $f_N(\edcedge{\alpha_{N+1}}{\beta_{N+1}}{\gamma_{N+1}}) = \edcedge{V_{m_{N+1}}}{{V_{n_{N+1}}}}{{V_{k_{N+1}}}}$ (or $f_N(\edcedge{\alpha_{N+1}}{\gamma_{N+1}}{\beta_{N+1}}) = \edcedge{V_{m_{N+1}}}{{V_{k_{N+1}}}}{{V_{n_{N+1}}}}$).
		Under the circumstances, there exists a sequence $\delta_0, \delta_1, \ldots , \delta_{M+1}$, where $\delta_l \in \eta_l$ for all $0 \le l \le M+1$ and $\delta_l \nu_{l,1} \nu_{l,2} \ldots \nu_{l,K_{l}} \delta_{l+1}$ is a path in $\mathcal{G}_N^c$. Consider the sequence, $\delta_{-1}, \delta_0, \delta_1, \ldots, \delta_{M+1}, \delta_{M+2}$, where $\delta_{-1} = \beta_{N+1}$ and $\delta_{M+2} = \gamma_{N+1}$. Also, $\delta_{-1}, \delta_0 \in \eta_0$, $\delta_{M+1}, \delta_{M+2} \in \eta_{M+1}$, and thus there are paths $\delta_{-1} \nu_{-1,1} \nu_{-1,2} \ldots \nu_{-1,K_{-1}} \delta_{0}$ and $\delta_{M+1} \nu_{M+1,1} \nu_{M+1,2} \ldots \nu_{M+1,K_{M+1}} \delta_{M+2}$. This implies $\edcedge{\alpha_{N+1}}{\gamma_{N+1}}{\gamma_{N+1}} \in E_c^{(N)}$ and thus $f_N(\edcedge{\alpha_{N+1}}{\gamma_{N+1}}{\gamma_{N+1}}) = \edcedge{V_{m_{N+1}}}{{V_{k_{N+1}}}}{{V_{k_{N+1}}}} \in  \widetilde{E_c}^{(N)}$, leading to a contradiction.
		
		Now we prove the theorem. For necessary condition, given any $V_i, V_j$ there exists a sequence $\delta_0, \delta_1, \ldots , \delta_{M+1}$, where $\delta_l \in \eta_l$ for all $0 \le l \le M+1$ and $\delta_l \nu_{l,1} \nu_{l,2} \ldots \nu_{l,K_{l}} \delta_{l+1}$ is a path in $\mathcal{G}_N^c$. Then $\overline{\mathcal{H}}^c$ is strongly connected. 
        
		For sufficient condition, given any $\alpha \in V_i \in \widetilde{V}, \beta \in V_j \in \widetilde{V}$, since $\overline{\mathcal{H}}^c$ is strongly connected, there exists a path $\eta_0 \eta_1 \ldots \eta_{M+1}$ such that $\eta_0 = V_i$ and $\eta_{M+1}=V_j$. 
		By surjective property of path, there is a sequence $\delta_{-1}, \delta_0, \delta_1, \ldots , \delta_{M+1} \delta_{M+2}$, where $\delta_{-1}, = \beta_{N+1}$ and $\delta_{M+2} = \gamma_{N+1}$ and there are path $\delta_{l}, \nu_{l,1} \nu_{l,2} \ldots \nu_{l,K_{l}} \delta_{l+1}$ for all $-1 \le l \le M+2$. It follows immediately that $\overline{\mathcal{G}}^c$ is strongly connected. This completes the proof.
\end{proof}

\begin{figure}
\begin{center}
\tikzset{
  r/.style = {>=stealth,draw, circle,fill = red!90!white, font=\bfseries, text =white},
  b/.style = {>=stealth,draw, circle,fill = blue!70!yellow, text = white},
  g/.style = {>=stealth,draw, circle,fill = green!70!yellow, text = white}
}
\tikzstyle{uni_arrow_edge} = [->, draw=black!90, line width=1]{}
\tikzstyle{div_arrow_edge} = [->, dashed, draw=black!90, line width=1]
\begin{tikzpicture}[
    scale = 0.6, transform shape, thick,
    every node/.style = {minimum size = 10mm},
    grow = down,  % alignment of characters
    level 1/.style = {sibling distance=2cm},
    level 2/.style = {sibling distance=3.5cm}, 
    level 3/.style = {sibling distance=1.5cm}, 
    level distance = 1.5cm,
  ]

\node[r] at (-2,0) (v3) {0};
\node[r] (v4) at (1,1.5) {1};
\node[r] (v5) at (1,-2) {2};

\draw[div_arrow_edge] (v3) .. controls (-2,1) and (-0.5,2) .. (v4);
\draw[uni_arrow_edge] (v3) .. controls (-0.5,0) and (0.5,0.5) .. (v4);

\draw[div_arrow_edge] (v3) .. controls (-1.5,-1.5) and (-0.5,-1.5) .. (v5);

\draw[uni_arrow_edge,draw=black]  (v4) edge (v5);
\draw[uni_arrow_edge] (v5) .. controls (0,-1) and (-0.5,-0.5) .. (v3);

\node[g] (v7) at (9,1.5) {3};
\node[g] (v8) at (9,-2) {4};

\draw[uni_arrow_edge]  (v7) edge (v8);
\draw[uni_arrow_edge] (v8) .. controls (10,-2) and (10,1.5) .. (v7);

\node[b] (v9) at (5,5) {5};

\draw[div_arrow_edge,line width=2] (v9) .. controls (3.5,5.5) and (1,3) .. (v4);
\draw[div_arrow_edge,line width=2] (v9) .. controls (5,4) and (3,-2) .. (v5);
\draw[uni_arrow_edge,line width=2] (v5) .. controls (1.5,-3.5) and (8.5,-3.5) .. (v8);
\draw[uni_arrow_edge,line width=2] (v7) .. controls (9,2.5) and (6,5.5) .. (v9);

\draw[draw=red]  (0,-0.3) ellipse (3 and 3);
\draw[draw=green]  (9,-0.1) ellipse (3 and 3);
\draw[draw=blue]  (v9) ellipse (2 and 2);
\end{tikzpicture}
\end{center}
\caption{The grouping $(d, c)$-reduction of the extended directed graph $\mathcal{G}$ in Example \ref{eg:dc-reduction-strong-connect} is an extended directed graph $\mathcal{H} = (\widetilde{V},\widetilde{E_c} ,\widetilde{E_d})$ with $\widetilde{V} = \{V_1, V_2, V_3\}$, $\widetilde{E_c} = \{(V_1, V_2), (V_2, V_3)\}$, and $\widetilde{E_d} = \{(V_3, V_1, V_1)\}$. It is seen immediately that $\overline{\mathcal{H}}^c$ is strongly connected.}
\label{fig:grouping-extgraph}
\end{figure}
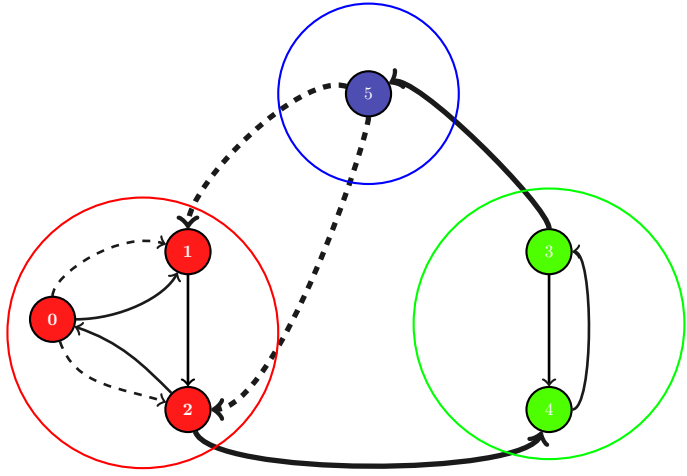

\begin{example}
Let $X$ be the TSFT considered in Example \ref{eg:dc-reduction-strong-connect}. Observe that $\mathcal{G}$ has three irreducible components seen as
$$
V = V_1 \bigcup V_2 \bigcup V_3, \quad \text{where} \quad V_1 = \{0, 1, 2\}, V_2 = \{3, 4\}, V_3 = \{5\}.
$$
Therefore, the grouping $(d, c)$-reduction of $\mathcal{G}$ is an extended directed graph $\mathcal{H}$ that consists of three vertices, two convergent-edges, and one divergent-edge. See Figure \ref{fig:grouping-extgraph}. It can be seen that the divergent-edge $(V_3, V_1, V_1)$ is actually a convergent-edge $(V_3, V_1)$, and $\overline{\mathcal{H}}^c$ is strongly connected.
\end{example}

% ------------------------------------------------------
\section{Conclusion and Discussion}

Since CPC-irreducibility of tree shifts of finite type implies the denseness of strongly periodic points, it is natural to elucidate the decidability of CPC-irreducible TSFTs. Theorems \ref{thm:dc-reduction-keep-irr} and \ref{thm:cpc-irr-extgraph-strong-connect} demonstrate that CPC-irreducibility of TSFTs is decidable. The related algorithm is referred to as the above flowchart in Figure \ref{fig:algorithm}.

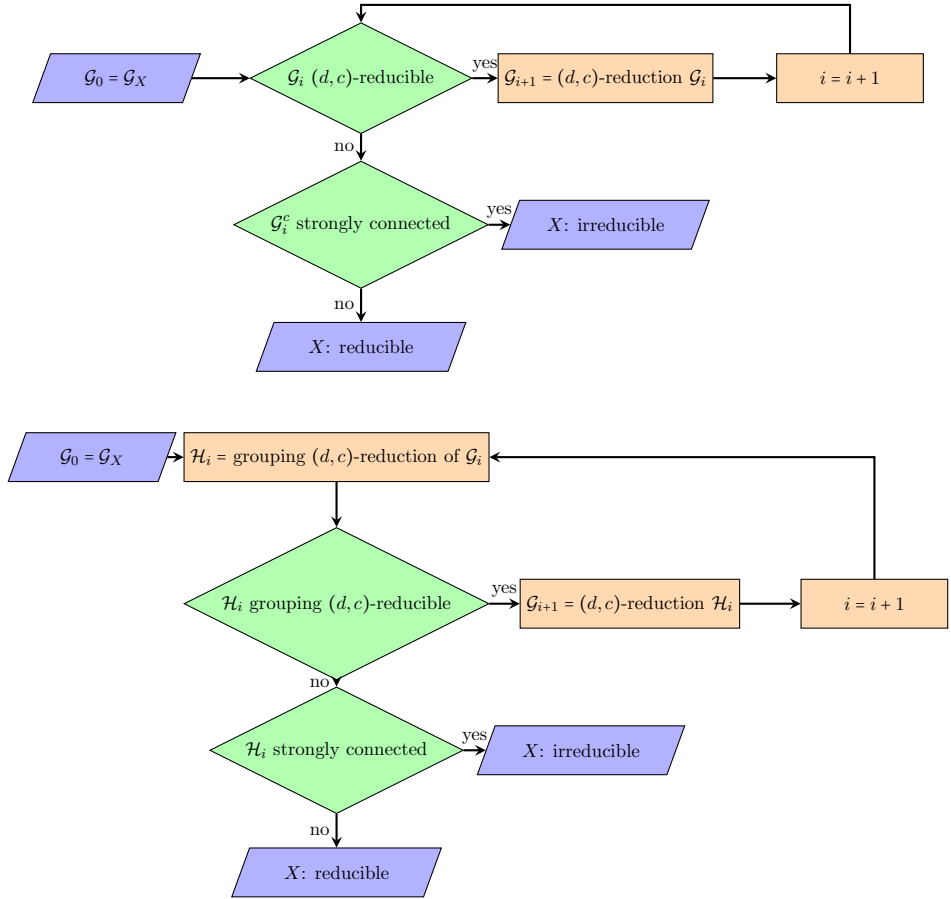
\begin{figure}
\begin{center}
\usetikzlibrary{shapes.geometric, arrows}
\tikzstyle{startstop} = [rectangle, rounded corners, minimum width=3cm, minimum height=1cm,text centered, draw=black, fill=red!30]
\tikzstyle{io} = [trapezium, trapezium left angle=70, trapezium right angle=110, minimum width=3cm, minimum height=1cm, text centered, draw=black, fill=blue!30,aspect=2]
\tikzstyle{process} = [rectangle, minimum width=3cm, minimum height=1cm, text centered, draw=black, fill=orange!30]
\tikzstyle{decision} = [diamond, minimum width=3cm, minimum height=1cm, text centered, draw=black, fill=green!30,aspect=2]
\tikzstyle{arrow} = [thick,->,>=stealth]
\begin{tikzpicture}[scale=0.65, transform shape,grow=left,node distance=5cm]
	\node (in1) [io] {$\mathcal{G}_0=\mathcal{G}_X$};
	\node (dec1) [decision, right of=in1] {$\mathcal{G}_i$ $(d,c)$-reducible};
	\node (dec2) [decision, below of=dec1, node distance=3cm] {$\mathcal{G}_i^c$ strongly connected};
	\node (pro1) [process, right of=dec1] {$\mathcal{G}_{i+1}=(d,c)$-reduction $\mathcal{G}_i$};
	\node (pro2) [process, right of=pro1] {$i=i+1$};
	
	\node (out1) [io, below of=dec2,inner xsep=-4pt, node distance=2.5cm] {$X$: reducible};
	\node (out2) [io, right of=dec2,inner xsep=-8pt] {$X$: irreducible};

	\draw [arrow] (in1) -- (dec1);
	\draw [arrow] (dec1) --node[anchor=east] {no} (dec2);
	\draw [arrow] (dec1) --node[anchor=south] {yes} (pro1);
	\draw [arrow] (pro1) -- (pro2);
	\draw [arrow] (dec2) --node[anchor=east] {no} (out1);
	\draw [arrow] (dec2) --node[anchor=south] {yes} (out2);

	\draw [arrow] (pro2) --++ (0cm,1.5cm) -| (dec1);
\end{tikzpicture}\vspace{2em}

\begin{tikzpicture}[scale=0.65, transform shape,grow=left,node distance=5cm]
	\node (in1) [io] {$\mathcal{G}_0=\mathcal{G}_X$};
	\node (pro0) [process, right of=in1] {$\mathcal{H}_i = $ grouping $(d,c)$-reduction of $\mathcal{G}_i$};
	\node (dec1) [decision, below of=pro0, node distance=3cm] {$\mathcal{H}_i$ grouping $(d,c)$-reducible};
	\node (dec2) [decision, below of=dec1, node distance=3cm] {$\mathcal{H}_i$ strongly connected};
	\node (pro1) [process, right of=dec1, node distance=6cm] {$\mathcal{G}_{i+1}=(d,c)$-reduction $\mathcal{H}_i$};
	\node (pro2) [process, right of=pro1] {$i=i+1$};
	
	\node (out1) [io, below of=dec2,inner xsep=-4pt, node distance=2.5cm] {$X$: reducible};
	\node (out2) [io, right of=dec2,inner xsep=-8pt] {$X$: irreducible};

	\draw [arrow] (in1) -- (pro0);
	\draw [arrow] (pro0) -- (dec1);
	\draw [arrow] (dec1) --node[anchor=east] {no} (dec2);
	\draw [arrow] (dec1) --node[anchor=south] {yes} (pro1);
	\draw [arrow] (pro1) -- (pro2);
	\draw [arrow] (dec2) --node[anchor=east] {no} (out1);
	\draw [arrow] (dec2) --node[anchor=south] {yes} (out2);

	\draw [arrow] (pro2) |-   (pro0);
\end{tikzpicture}
\end{center}
\caption{Flowcharts of algorithms of $(d, c)$-reduction and grouping $(d, c)$-reduction of extended directed graph (Theorems \ref{thm:cpc-irr-extgraph-strong-connect} and \ref{thm:group-dcreduction-CPCirr}) that demonstrate the CPC-irreducibility of tree shifts of finite type is decidable.}
\label{fig:algorithm}
\end{figure}

Whenever a considered TSFT is complicated, for instance, the alphabet is large, or the forbidden set is small, Theorem \ref{thm:group-dcreduction-CPCirr} provides a more efficient algorithm for determining if it is CPC-irreducible. See the below flowchart in Figure \ref{fig:algorithm}. On the other hand, the following question is interesting and remains open.

\begin{problem}
Is the CPC-irreducibility of sofic tree shifts decidable?
\end{problem}

% ------------------------------------------------------
\section*{Acknowledgment}

The authors want to express their gratitude to the anonymous referees for their valuable comments and suggestions, which significantly improve the quality of this paper and make the paper more readable.

% ------------------------------------------------------
\section*{Appendix. The Complexity of Algorithm}

In this appendix, we present the pseudo code of our algorithm and estimate the complexity of the algorithm.

The index start from 0 for the following pseudocode.
\IncMargin{1em}
\begin{algorithm}
\caption{(d,c)-reduction}
\SetKwData{foundReachable}{foundReachable}\SetKwData{divEdge}{divEdge}
\SetKwFunction{isReachable}{isReachable}\SetKwFunction{delete}{delete}\SetKwFunction{append}{append}\SetKwFunction{size}{size}\SetKwFunction{dcReduction}{dcReduction}
\SetKwInOut{Input}{input}\SetKwInOut{Output}{output}
\SetKwProg{Fn}{Function}{}{end}

\Input{
	\vspace{-0.2em}The object \DataSty{graph} has the following data member: 
\begin{tabular}[t]{ l  l  l }	
	\DataSty{graph.v}: &  list of int  \\
	\DataSty{graph.ce}: &  list of int[2]   \\
	\DataSty{graph.de}: &  list of int[3]   \\
\end{tabular}
}

\Output{CPC-irreducible or not}
\BlankLine

\Fn{\dcReduction{\DataSty{graph}}}{
	\foundReachable = True\;
	\While{\foundReachable}{
		\foundReachable = False\;
		$i=0$\;
		\While{$i$ < \size{\DataSty{graph.de}}}{
			\divEdge = \DataSty{graph.de.divEdge[i]} \;
			\If{\isReachable{\DataSty{divEdge[1]},\DataSty{divEdge[2]}}}{
				\DataSty{divEdge[1]} = \DataSty{divEdge[2]}\;
				\DataSty{graph.de}.\delete{$i$}\;
				\DataSty{graph.ce}.\append(\divEdge)\;
				\foundReachable = True\;
			}
			\ElseIf{\isReachable{\DataSty{divEdge[2]}, \DataSty{divEdge[1]}}}{
				\DataSty{divEdge[2]} = \DataSty{divEdge[1]}\;
				\DataSty{graph.de}.\delete{$i$}\;
				\DataSty{graph.ce}.\append{\divEdge}\;
				\foundReachable = True\;
			}
			\Else{
				$i++$\;
			}
		}
	}
	%\Return isStronglyConnected(g)
}
\end{algorithm}
\DecMargin{1em}

Suppose there are $k$ convergent-edges, $m$ divergent-edges, and $n$ vertices in the extended directed graph. It is seen that the complexity of ``isReachable" part is at most $O(m+n+k)$. Since we have $m$ divergent-edges, the complexity of our algorithm is at most
$$
O((m+n+k) \cdot (m+(m-1)+...+1))=O\left((m+n+k) \cdot \dfrac{m(m-1)}{2}\right) = O(m^3).
$$

% bibliography ---------------------------------------------------

% -------------------------------------------------------------
\end{document}